\documentclass[11pt]{amsart}
\usepackage[a4paper,margin=29mm]{geometry}
\usepackage[T1]{fontenc}
\usepackage{lmodern}
\usepackage{microtype}
\usepackage{indentfirst}
\usepackage{graphicx}
\usepackage{amsmath,amssymb,amsthm,mathtools}
\usepackage{xcolor}
\usepackage[
  colorlinks=false,
  pdfborder={0 0 1},
  linkbordercolor=red,
  citebordercolor=green,
  urlbordercolor=blue
]{hyperref}

\usepackage{etoolbox}

\makeatletter

\patchcmd{\thebibliography}
  {\usecounter{enumiv}}
  {\usecounter{enumiv}%
   \setlength{\itemsep}{4pt}%
   \setlength{\parsep}{0pt}}
  {}{}

\makeatother

\allowdisplaybreaks
\newtheorem{theorem}{Theorem}[section]
\newtheorem{lemma}[theorem]{Lemma}
\newtheorem{proposition}[theorem]{Proposition}
\newtheorem{definition}[theorem]{Definition}

\newtheorem{corollary}[theorem]{Corollary}
\theoremstyle{remark}
\newtheorem{remark}[theorem]{Remark}
\newcommand{\cM}{\mathcal M}
\newcommand{\cN}{\mathcal N}
\newcommand{\cF}{\mathcal F}
\newcommand{\Gm}{G_{\mathbf m}}

\newcommand{\N}{\mathbb N}

\newcommand{\1}{\mathbf 1}
\newcommand{\norm}[1]{\left\lVert #1\right\rVert}
\newcommand{\dunion}{\mathbin{\dot\cup}}
\newcommand{\bigdunion}{\mathop{\dot\bigcup}}

\begin{document}

 \baselineskip 16.6pt
\hfuzz=6pt

\widowpenalty=10000

\renewcommand{\theequation}
{\thesection.\arabic{equation}}

\def\SL{\sqrt H}

\newcommand{\mar}[1]{{\marginpar{\sffamily{\scriptsize
        #1}}}}

\newcommand{\as}[1]{{\mar{AS:#1}}}

\numberwithin{equation}{section}

\title[Vector-valued Vilenkin LPR Inequalities]
{Vector-valued Littlewood--Paley--Rubio de Francia Inequalities on Vilenkin Systems}

\author{Deyu Chen}
\address{Deyu Chen, Institute for Advanced Study in Mathematics,
Harbin Institute of Technology, Harbin 150001, China}
\email{1201200317@stu.hit.edu.cn}

\author{Guixiang Hong}
\address{Guixiang Hong, Institute for Advanced Study in Mathematics,
Harbin Institute of Technology, Harbin 150001, China}
\email{gxhong@hit.edu.cn}

\date{}



\begin{abstract}
We investigate two natural formulations of equal-length
Littlewood--Paley--Rubio de Francia (LPR) inequalities on arbitrary Vilenkin systems
\(\Gm\). The canonical formulation uses canonical intervals
in \(\mathbb N\) of equal cardinality, while Vilenkin's formulation
uses translates of a common initial interval under the group
operation on the dual of \(\Gm\).

In this paper, our main results show that the definitions of these two types of intervals exhibit dramatically different properties. More precisely,
for the canonical formulation, we
construct a counterexample showing that, for every \(2<p<\infty\),
the Schatten class \(S^p\) fails the equal-length LPR property, in
sharp contrast to the corresponding result on the torus; for
Vilenkin's formulation, we prove that, for every \(2\le p<\infty\),
a Banach space \(X\) has the equal-length LPR property if and only
if it is UMD and has type~\(2\). The sufficiency proof combines a
transference argument from the torus to finite cyclic groups with
an embedding of Vilenkin intervals into digital rectangular
envelopes, yielding bounds independent of the generating sequence.
\end{abstract}

\maketitle

% \clearpage
% \begingroup
% \small
% \tableofcontents
% \endgroup
% \clearpage

\section{Introduction}
\subsection{Background}
\subsubsection{Littlewood--Paley--Rubio de Francia inequalities
on the torus}

In his seminal work, Rubio de Francia \cite{RubioDeFrancia1985}
proved that, for every finite family \(\{I_s\}_s\) of pairwise
disjoint intervals in \(\mathbb R\), every \(2\le p<\infty\), and
every \(f\in L^p(\mathbb R)\),
\begin{equation}\label{eq:scalar-R-lpr}
 \norm{
  \left(\sum_s |P_{I_s}^{\mathbb R}f|^2\right)^{1/2}
 }_{L^p(\mathbb R)}
 \le
 C_p\norm{f}_{L^p(\mathbb R)},
\end{equation}
where \(P_{I_s}^{\mathbb R}\) denotes the Fourier multiplier with
symbol \(\1_{I_s}\), and \(C_p\) depends only on \(p\). This is the classical
Littlewood--Paley--Rubio de Francia inequality, which we abbreviate
as the LPR inequality.
By the Khintchine inequality,
\eqref{eq:scalar-R-lpr} is equivalent, up to a change of constant, to
the randomized estimate
\begin{equation}\label{eq:scalar-R-lpr2}
 \left(
  \mathbb E_{\varepsilon}
  \norm{\sum_s\varepsilon_sP_{I_s}^{\mathbb R}f}_{L^p(\mathbb R)}^p
 \right)^{1/p}
 \le
 C_p\norm{f}_{L^p(\mathbb R)},
\end{equation}
where \((\varepsilon_s)_s\) is an i.i.d.\ Rademacher sequence
on an auxiliary probability space.
A standard transference argument yields the corresponding inequality
on the torus
\begin{equation}\label{eq:scalar-torus-lpr2}
 \left(
  \mathbb E_{\varepsilon}
  \norm{\sum_s\varepsilon_sP_{I_s}^{\mathbb T}f}_{L^p(\mathbb T)}^p
 \right)^{1/p}
 \le
 C_p\norm{f}_{L^p(\mathbb T)}
\end{equation}
for every \(2\le p<\infty\), every finite family \(\{I_s\}_s\)
of pairwise disjoint intervals in \(\mathbb Z\), and every
\(f\in L^p(\mathbb T)\). Here \(P_{I_s}^{\mathbb T}\) denotes
the Fourier projection associated with \(I_s\).

Vector-valued extensions of LPR inequalities depend much more delicately on the geometry
of the range space. Motivated by the vector-valued analogue of
\eqref{eq:scalar-R-lpr2}, Berkson, Gillespie, and Torrea
\cite{BGT} introduced the LPR property to describe vector-valued
extensions of \eqref{eq:scalar-R-lpr2}. For \(2\le p<\infty\),
a Banach space \(X\) is said to have the
\(\operatorname{LPR}_{p}^{\mathbb R}\) property if there exists a
constant \(C_{p,X}>0\) such that, for every finite family
\(\{I_s\}_s\) of pairwise disjoint intervals in \(\mathbb R\) and
every \(f\in L^p(\mathbb R;X)\),
\begin{equation}\label{LPR}
 \left(
  \mathbb E_{\varepsilon}
  \norm{\sum_s\varepsilon_sP_{I_s}^{\mathbb R}f}_{L^p(\mathbb R;X)}^p
 \right)^{1/p}
 \le
 C_{p,X}\norm{f}_{L^p(\mathbb R;X)}.
\end{equation}
The corresponding property on the torus, denoted by
\(\operatorname{LPR}_{p}^{\mathbb T}\), is defined analogously,
using intervals in \(\mathbb Z\). Restricting these intervals to
have equal length gives the weaker property
\(\operatorname{LPR}_{p,=}^{\mathbb T}\). Hyt\"onen, Torrea, and
Yakubovich \cite{HTY} proved that, for every \(2\le p<\infty\),
a Banach space \(X\) has the
\(\operatorname{LPR}_{p,=}^{\mathbb T}\) property if and only if
\(X\) is UMD and has type~\(2\). This characterizes the class of UMD Banach spaces of type~\(2\).

It remains open whether the UMD property together with type~\(2\)
is sufficient for the full \(\operatorname{LPR}_{p}^{\mathbb T}\)
property, which allows disjoint intervals of arbitrary lengths.
In particular, it is a longstanding open problem
whether the Schatten class \(S^p\) has the
\(\operatorname{LPR}_{p}^{\mathbb T}\) property for
\(2<p<\infty\); see \cite[p.~732]{ParcetICM2026}.
Important positive results are available for certain Banach
lattices. Potapov, Sukochev, and Xu
\cite{PotapovSukochevXu2012} proved that if \(X\) is a Banach
lattice whose \(2\)-concavification \(X_{(2)}\) is UMD, then
\(X\) has the \(\operatorname{LPR}_{p}^{\mathbb T}\) property
for every \(2<p<\infty\). Their proof relies on the pointwise
lattice structure and suitable maximal-function estimates,
and therefore does not directly extend to noncommutative
spaces without such a structure.

\subsubsection{Littlewood--Paley--Rubio de Francia inequalities
for Vilenkin systems}

We now turn to LPR inequalities for Vilenkin systems. Let
\(\mathbf m=(m_k)_{k\ge0}\) be a sequence of integers with \(m_k\ge2\),
and let
\[
 \Gm=\prod_{k\ge0}\mathbb Z_{m_k}
\]
be the associated Vilenkin group, equipped with its normalized Haar
measure \(\mu\). The Walsh group $G_{\mathbf 2}$ corresponds to the special case $m_k=2$ for all $k\ge 0$. Let \((\psi_n)_{n\ge0}\) denote the
associated Vilenkin character system. For an \(X\)-valued Vilenkin
polynomial \(f\), set
\[
 \widehat f(n)
 :=
 \int_{\Gm}f(x)\overline{\psi_n(x)}\,d\mu(x),
\]
and for \(I\subset\mathbb N\), define the Fourier--Vilenkin
projection
\[
 P_If
 :=
 \sum_{n\in I}\widehat f(n)\psi_n.
\]

Harmonic analysis on Vilenkin groups has been extensively studied;
see, for example,
\cite{Young1976,Young1993,Young1994,AvdispahicMemicWeisz2012}
for the scalar theory and
\cite{ClementDePagterSukochevWitvliet2000,Weisz2007,CH2026}
for vector-valued results. In particular, in \cite{CH2026} we established a sampling
factorization identity linking Fourier multipliers on the torus
to those on finite cyclic groups, and used this connection to
obtain uniform bounds for UMD-valued Fourier--Vilenkin partial-sum
operators. This connection motivates the present study of the
relationship between LPR inequalities on the torus and on
Vilenkin groups. We begin by recalling the scalar theory.
In the scalar setting, the natural
question is whether the estimate
\begin{equation}\label{eq:scalarlprm}
 \norm{
  \left(\sum_s |P_{I_s}f|^2\right)^{1/2}
 }_{L^p(\Gm)}
 \le
 C_{p,\mathbf m}\norm{f}_{L^p(\Gm)}
\end{equation}
holds for every finite family \(\{I_s\}_s\) of pairwise disjoint
intervals in \(\mathbb N\), every \(2\le p<\infty\), and every
\(f\in L^p(\Gm)\), with a constant independent of the family and \(f\). A further question is whether the constant can be chosen
independently of the generating sequence \(\mathbf m\). 
 Osipov \cite{Osipov} first established
the scalar LPR inequality for the Walsh system.
Tselishchev \cite{Tselishchev2021} subsequently extended this result to bounded Vilenkin
systems, with a constant allowed to depend on
\(\sup_km_k\). He \cite{Tselishchev2024} later removed the boundedness assumption and proved
\eqref{eq:scalarlprm} for arbitrary Vilenkin systems with a constant
\(C_p\) independent of \(\mathbf m\).  

Both Tselishchev's argument and our approach in~\cite{CH2026}
use finite cyclic groups to connect classical Fourier estimates
with Vilenkin analysis, with bounds uniform in the group orders.
The main difference is that Tselishchev relies on scalar weighted
Littlewood--Paley estimates transferred along
\(\mathbb R\to\mathbb Z\to\mathbb Z_m\), with the last step
implemented through a Poisson-kernel limiting argument.
He then uses a Whitney decomposition and cyclic kernel estimates
to pass from these cyclic estimates to inequalities on Vilenkin
systems. Our approach instead uses an exact sampling factorization
to derive uniform estimates for cyclic Fourier projections
from torus multiplier estimates. We then combine the Paley
conjugation identity with martingale decoupling to pass from
these cyclic estimates to inequalities on Vilenkin systems.

We next introduce the corresponding vector-valued property.
\begin{definition}\label{def1}
 Let \(2\le p<\infty\). A Banach space \(X\) is said to have the
 \(\operatorname{LPR}_p\) property if there exists a constant
 \(C_{p,X,\mathbf m}>0\) such that, for every finite family
 \(\{I_s\}_s\) of pairwise disjoint intervals in \(\mathbb N\) and
 every \(f\in L^p(\Gm;X)\),
 \begin{equation}\label{LPRG}
  \left(
   \mathbb E_{\varepsilon}
   \norm{\sum_s\varepsilon_sP_{I_s}f}_{L^p(\Gm;X)}^p
  \right)^{1/p}
  \le
  C_{p,X,\mathbf m}\norm{f}_{L^p(\Gm;X)}.
 \end{equation}
 Here \((\varepsilon_s)_s\) is an i.i.d.\ Rademacher sequence
 defined on an auxiliary probability space.
 % If the constant in \eqref{LPRG} can be chosen independently of the
 % generating sequence \(\mathbf m\), we say that \(X\) has the uniform
 % Vilenkin \(\operatorname{LPR}_p\) property.
\end{definition}

The vector-valued theory is less developed than its scalar
counterpart. For the Walsh system, Tselishchev \cite{TselishchevVector2021} obtained a vector-valued extension of the main results of \cite{PotapovSukochevXu2012}, showing that if \(X\) is
a Banach lattice whose \(2\)-concavification \(X_{(2)}\) is UMD,
then \(X\) has the \(\operatorname{LPR}_p\) property for every
\(2<p<\infty\). The maximal-function approach used
in this result may also extend to bounded Vilenkin systems;
in its current form, however, it does not provide estimates
uniform over arbitrary generating sequences. Beyond this
Banach-lattice result for the Walsh system, the
\(\operatorname{LPR}_p\) property for general Vilenkin systems
and broader classes of Banach spaces remains largely open.

\subsection{Further Motivations and Aims}
Our starting point is the sampling factorization identity established
in our earlier work~\cite{CH2026}. This identity links Fourier
multipliers on the torus with those on finite cyclic groups and,
through the cyclic-coordinate structure of \(\Gm\), provides a connection
with Fourier--Vilenkin multipliers. This connection motivates us to
investigate the relationship between LPR inequalities on the torus and
on Vilenkin groups. Such a comparison may help clarify the scope and
appropriate formulation of possible transference principles between
Fourier multipliers in these two settings.

We focus on the equal-length LPR inequality studied by Hyt\"onen,
Torrea, and Yakubovich~\cite{HTY}. Their theorem characterizes the
equal-length LPR property on the torus by the conjunction of the UMD
property and type~\(2\). We ask whether an analogous characterization
holds on Vilenkin groups and, in particular, which notion of a
frequency interval is compatible with this analogy. On the torus,
equal-length frequency intervals are canonical translates of a fixed
interval in \(\mathbb Z\). In the Vilenkin setting, the usual ordering
of the frequency indices and the group operation on the dual give
rise to two distinct formulations, which we now introduce.

Given a generating sequence
\(\mathbf m=(m_k)_{k\ge0}\), define
\begin{equation}\label{Mk}
M_0:=1,
\qquad
M_{k+1}:=m_kM_k,
\qquad k\ge0.
\end{equation}
Every integer $n\in\mathbb N$ can be uniquely represented by a mixed-radix expansion
\[
n=\sum_{k\ge0}n_kM_k,
\qquad 0\le n_k<m_k,
\]
with only finitely many nonzero digits.
For \(a,b\in\N\), define
\[
a\oplus b
 :=\sum_k \bigl((a_k+b_k)\bmod m_k\bigr)M_k,
\quad
a\ominus b
 :=\sum_k \bigl((a_k-b_k)\bmod m_k\bigr)M_k.
\]
Then \((\N,\oplus)\) is naturally identified with the dual group of \(\Gm\). For \(D\subset\N\), we write
\[
a\oplus D:=\{a\oplus n:n\in D\},
\qquad
D\ominus a:=\{n\ominus a:n\in D\},
\qquad
a\ominus D:=\{a\ominus n:n\in D\}.
\]

In this setting, we distinguish two formulations of frequency intervals
and equal length, which we call the canonical formulation and
Vilenkin's formulation, respectively.

\begin{enumerate}
\item[(i)] We disregard the group structure on the dual group \((\N,\oplus)\) and use the usual notion of an interval in \(\N\). Since \(\N\) is discrete, two intervals have equal length if and only if they have the same cardinality.

\item[(ii)] We take the group structure of \((\N,\oplus)\) into account by replacing the usual addition $+$ with the group addition \(\oplus\). A set \(I\subset\N\) is called a Vilenkin interval if
\[
I=\zeta\oplus[0,L)
\]
for some \(\zeta,L\in\N\). Two Vilenkin intervals \(I_1\) and \(I_2\) are said to have equal length if there exist \(\zeta_1,\zeta_2,L\in\N\) such that
\[
I_i=\zeta_i\oplus[0,L),\qquad i=1,2.
\]

\end{enumerate}

\begin{remark}
These two formulations are not equivalent since a Vilenkin interval may not be an interval on $\mathbb N$. For example, in the Walsh case, a direct calculation gives
\[
2\oplus[0,3)=\{0,2,3\},
\]
which is not a canonical interval.
\end{remark}

% \begin{remark}\label{remark1}
%     In previous studies, Osipov and Tselishchev both adopted the canonical formulation to define the scalar and vector-valued LPR inequality. If one adopts Vilenkin's formulation, one obtains a new version of \eqref{LPRG} and of the LPR property in Definition \ref{def1}. We conjecture that the previous results of Osipov and Tselishchev can also be generalized to this new version of Vilenkin's formulation.
% \end{remark}

The two formulations lead to the following distinct versions of the equal-length LPR property.

\begin{definition}
A Banach space \(X\) is said to have the
\(\mathrm{LPR}_{p,=}\) property if \eqref{LPRG} holds for every
pairwise disjoint family \((I_s)_s\) of canonical intervals in \(\N\)
having equal cardinality.

A Banach space \(X\) is said to have the
\(\mathrm{LPR}_{p,=}^{\oplus}\) property if \eqref{LPRG} holds for
every pairwise disjoint family \((I_s^\oplus)_s\) of Vilenkin intervals
of the form
\[
I_s^\oplus=\zeta_s\oplus[0,L),
\qquad \zeta_s,\,L\in\N,
\]
where the parameter \(L\) is independent of \(s\).
\end{definition}
\begin{remark}
One may also consider a more general version of Vilenkin's formulation,
in which the intervals are of the form
\[
I_s^\oplus
=
\zeta_s\oplus[L_1,L_2),
\qquad
\zeta_s,L_1,L_2\in\N.
\]
This formulation is in fact equivalent to the original one defining the
\(\mathrm{LPR}_{p,=}^{\oplus}\) property; see
Corollary~\ref{cor:oplus-common-interval}.
\end{remark}

The canonical formulation is the equal-length restriction of
Definition~\ref{def1} and follows the convention used in the existing
LPR theory for Vilenkin systems. Vilenkin's formulation instead
emphasizes translation in the dual group: the sets
\(\zeta_s\oplus[0,L)\) are translates of a common frequency set under
the operation that corresponds to modulation by Vilenkin characters.
Comparing these formulations therefore allows us to examine both the
Banach-space geometry required for the equal-length LPR inequality and
the interval structure that a transference principle connecting the
torus and Vilenkin settings should respect.

% Similarly, we can also define the full LPR property for Vilenkin intervals, which is different from Definition \ref{def1}.
% \begin{definition}
% A Banach space \(X\) is said to have the
% \(\mathrm{LPR}_{p}^{\oplus}\) property if \eqref{LPRG} holds for
% every pairwise disjoint family \((I_s^\oplus)_s\) of Vilenkin intervals
% of the form
% \[
% I_s^\oplus=\zeta_s\oplus[0,L_s),
% \qquad \zeta_s,L_s\in\N.
% \]
% \end{definition}

\subsection{Main results}
We establish the following results for Vilenkin systems. For the two
formulations of the equal-length LPR property, we first show that either
formulation implies that $X$ is UMD and has type~$2$, exactly as in the
torus case~\cite{HTY}.

\begin{theorem}\label{thm:walsh-necessity}
Let \(2\le p<\infty\), and let
\(\mathbf m=(m_k)_{k\ge0}\) be any generating sequence with \(m_k\ge2\).
If a Banach space \(X\) satisfies the
\(\operatorname{LPR}_{p,=}\) property, then \(X\) is $\operatorname{UMD}$ and has type~\(2\).
The same conclusion holds if \(X\) satisfies the
\(\operatorname{LPR}_{p,=}^{\oplus}\) property.
\end{theorem}

The following theorem shows that the converse of Theorem~\ref{thm:walsh-necessity} fails for the \(\operatorname{LPR}_{p,=}\) property.

\begin{theorem}\label{thm:counterexample-main}
Let \(2<p<\infty\), and let
\(\mathbf m=(m_k)_{k\ge0}\) be any generating sequence with \(m_k\ge2\).
Set
\[
 \cN_{\mathbf m}=L^ \infty(\Gm)\,\overline\otimes\,B(\ell^2),
\]
endowed with the tensor-product trace. Then neither of the following
estimates holds with a finite constant \(C\), uniformly for all
\(f\in L^p(\cN_{\mathbf m})\) and all finite families
\((I_s)_s\) of pairwise disjoint intervals of equal length in \(\N\):
\begin{equation}\label{eq:false-operator-rubio}
 \norm{\left(\sum_s |P_{I_s}f|^2\right)^{1/2}}_{L^p(\cN_{\mathbf m})}
\le
C\norm{f}_{L^p(\cN_{\mathbf m})},
\end{equation}
and
\begin{equation}\label{eq:false-operator-rubio-row}
 \norm{\left(\sum_s |(P_{I_s}f)^*|^2\right)^{1/2}}_{L^p(\cN_{\mathbf m})}
\le
C\norm{f}_{L^p(\cN_{\mathbf m})}.
\end{equation}
Consequently, by the noncommutative Khinchine inequality, for any $2<p<\infty,$ the Schatten class $S^p(B(\ell^2))$ does not satisfy the $\operatorname{LPR}_{p,=}$ property.
\end{theorem}

While $\operatorname{UMD}$ Banach spaces of type $2$ may fail to satisfy
the $\operatorname{LPR}_{p,=}$ property, we show that they always satisfy
the $\operatorname{LPR}_{p,=}^\oplus$ property, with a constant independent
of the generating sequence $\mathbf m$.

\begin{theorem}
\label{thm:oplus-equal-umd-type2-main}
Let \(2\le p<\infty\), let \(X\) be a $\operatorname{UMD}$ Banach space with type~\(2\), and
let \(\mathbf m=(m_k)_{k\ge0}\) be any generating sequence with
\(m_k\ge2\). Then \(X\) satisfies the
\(\operatorname{LPR}_{p,=}^{\oplus}\) property with the constant $$C_{p,X,\mathbf m}\lesssim_p t_{2,X}^2\beta_{p,X}^6$$ independent of $\mathbf m.$
\end{theorem}

Here $t_{2,X}$ (resp. $\beta_{p,X}$) is the type $2$ constant (resp. UMD constant) that will be introduced in the next section.
The proof of Theorem~\ref{thm:oplus-equal-umd-type2-main} shares with
the torus argument of Hyt\"onen, Torrea, and Yakubovich~\cite{HTY}
the use of type~\(2\) to control randomized Fourier coefficients and
of the UMD property to recover sharp frequency projections. In their
proof, this is implemented through de la Vall\'ee Poussin
approximations and the boundedness of the Hilbert transform. Our
argument adapts this strategy to the mixed-radix geometry of Vilenkin
frequencies:
\begin{enumerate}
\item[\textup{(i)}]
We first derive a randomized Fourier-coefficient estimate from
the type~\(2\) property in
Lemma~\ref{lem:type2-randomized-coefficients}.
Applying this estimate to Vilenkin characters yields the LPR inequality for distinct cosets of a common
finite subgroup of the dual in
Lemma~\ref{lem:oplus-coset-randomization}.

\item[\textup{(ii)}]
In Lemma~\ref{lem:two-parameter-cyclic-equal}, we combine
the torus equal-length LPR theorem
\cite[Theorem~1.1]{HTY} with a modulation argument and the
sampling factorization from~\cite{CH2026} to establish a
two-parameter randomized estimate for equal-length cyclic
interval projections, with constants independent of the
order of the cyclic group.

\item[\textup{(iii)}]
The main geometric ingredient is
Lemma~\ref{lem:digital-rectangular-envelopes}, which embeds
disjoint Vilenkin intervals of equal length into digital
rectangular envelopes that can be partitioned into three
families of pairwise disjoint sets.
Their product structure separates the relevant cyclic
coordinate from the remaining digits, allowing us to combine
the cyclic estimate with the coset estimate. This yields
the LPR inequality for the envelopes in
Lemma~\ref{lem:rectangular-envelope-estimate}.

\item[\textup{(iv)}]
Finally, Proposition~\ref{prop:oplus-reduction} uses the modulation identity to reduce the desired LPR inequality to the uniform boundedness
of partial-sum operators on
\(L^p(\Gm;X)\).
Since \(X\) is UMD,
\cite[Theorem~1.2]{CH2026} supplies this bound uniformly
in the generating sequence, completing the proof of
Theorem~\ref{thm:oplus-equal-umd-type2-main}.
\end{enumerate}

As a consequence of Theorem~\ref{thm:oplus-equal-umd-type2-main}, one
also obtains an LPR inequality for the more general family of Vilenkin intervals
\[
I_s^{\oplus}
=
\zeta_s\oplus[L_1,L_2),
\]
where \(0\le L_1<L_2\) are independent of
\(s\) and $\zeta_s,L_1,L_2\in\mathbb N$.

\begin{corollary}\label{cor:oplus-common-interval}
Under the hypotheses of
Theorem~\ref{thm:oplus-equal-umd-type2-main}, the inequality
\eqref{LPRG} holds for every finite family of pairwise disjoint sets
\[
I_s^{\oplus}
=
\zeta_s\oplus[L_1,L_2),
\qquad
\zeta_s\in\mathbb N, \quad L_1,L_2\in\N, \quad 0\le L_1<L_2,
\]
with the constant \[C_{p,X,\mathbf m}\lesssim_p t_{2,X}^2\beta_{p,X}^6\] independent of $\mathbf m.$
\end{corollary}

The preceding results show that the two formulations of LPR inequality have sharply
different behavior. 
This distinction suggests that a transference principle intended to
relate the equal-length LPR theories on the torus and on Vilenkin
groups should preserve a correspondence between torus frequency
intervals and Vilenkin intervals, together with their respective
translation structures. 
% Corollary~\ref{cor:oplus-common-interval}
% reinforces this viewpoint by allowing \(\oplus\)-translates of any
% fixed canonical interval. 
These observations motivate the search for
a multiplier transference principle that respects this interval
correspondence.

% \subsection{Organization}

% The remainder of the paper is organized as follows. Section~\ref{s2} collects
% the necessary preliminaries on Vilenkin systems, vector-valued geometry,
% noncommutative \(L^p\)-spaces, and row--column sequence spaces. Section~\ref{s3}
% proves the necessity of the UMD and type~\(2\) conditions for both
% formulations of the equal-length LPR property.
%  Section~\ref{s4} constructs the
% operator-valued counterexample showing that the Schatten class \(S^p\),
% \(2<p<\infty\), fails the \(\mathrm{LPR}_{p,=}\) property.
%  In Section~\ref{s5}, we prove that any UMD Banach space with type $2$ satisfies the
% \(\mathrm{LPR}_{p,=}^{\oplus}\) property by combining a type~\(2\)
% randomized Fourier-coefficient estimate, a two-parameter estimate on
% finite cyclic groups, an embedding of Vilenkin intervals into digital rectangular envelopes, and the uniform
% boundedness of partial-sum operators given in \cite[Theorem 1.2]{CH2026};
% we then prove Corollary~\ref{cor:oplus-common-interval}. 

\subsection{Notation}
Throughout the paper, we set
\(
\mathbb N:=\{0,1,2,\dots\}.
\)
We identify the torus with the unit circle, via the parametrization
\(t\mapsto e^{2\pi i t}\) for \(t\in[0,1)\).
For \(m\ge2\), we write
\(
\mathbb Z_m:=\mathbb Z/m\mathbb Z
\)
for the cyclic group of order \(m\), endowed with normalized counting measure.

For \(a<b\), the notation \([a,b)\) denotes the corresponding
half-open interval. When the endpoints are integers, it is understood
as a subset of \(\mathbb Z,\mathbb N\) or \(\mathbb Z_m\), according
to the context. For \(1<p<\infty\), \(p'\) denotes the conjugate
exponent, namely
\(
1/p+1/p'=1.
\)
The symbol \(\mathbf 1_E\) denotes the indicator function of a set
\(E\), and \(\operatorname{supp} f\) denotes the support of \(f\).
We write \(\overline z\) for the complex conjugate of \(z\),
\(A^*\) for the adjoint of an operator \(A\), and
\(A^{\mathsf T}\) for its transpose.

For nonnegative quantities \(A\) and \(B\), the notation
\(
A\lesssim_{\Theta} B
\)
means that \(A\le C_{\Theta}B\) for a constant \(C_{\Theta}\)
depending only on the parameters listed in \(\Theta\).
We use \(\gtrsim_{\Theta}\) for the reverse inequality.
The dependence of implicit constants is indicated by subscripts
whenever necessary.

The notation \(L^p(\Omega;X)\) stands for the Bochner
\(L^p\)-space of \(X\)-valued functions. Expectations and
probabilities are denoted by \(\mathbb E\) and \(\mathbb P\),
respectively, with subscripts indicating the underlying random
variables. The sequence \((\varepsilon_j)_j\) denotes an i.i.d.\
Rademacher sequence on an auxiliary probability space $(\Omega_\varepsilon,\mathbb P)$.

We denote by \(\ell^2\) the Hilbert space of square-summable
sequences and by \(B(\ell^2)\) the algebra of bounded operators
on \(\ell^2\). For \(1\le p<\infty\), \(S^p=S^p(B(\ell^2))\)
denotes the Schatten \(p\)-class, equipped with the norm
\(
\|A\|_{S^p}
=
\bigl(\operatorname{Tr}(|A|^p)\bigr)^{1/p}.
\)
Furthermore, \(\ell^2_N=\mathbb C^N\),
\(\mathbb M_N=B(\ell^2_N)\), and \(e_{ij}\) denotes the standard
matrix unit. We write \(S^p_N\) for \(\mathbb M_N\) equipped with the
Schatten \(p\)-norm defined using the unnormalized matrix trace.
The symbol \(\overline\otimes\) denotes the spatial tensor product of
von Neumann algebras.

\section{Preliminaries}\label{s2}

We recall the Vilenkin character system, the canonical filtration,
and the associated Fourier projections. We then introduce the
Banach-space notions and noncommutative column and row sequence
spaces needed in the subsequent estimates.

\subsection{Basic definitions for Vilenkin systems}\label{s2.1}

Every \(n\in\N\) has a unique mixed-radix expansion
\[
 n=\sum_{k\ge0}n_kM_k,
 \qquad
 0\le n_k<m_k,
\]
with only finitely many nonzero digits, where \(M_k\) is defined
in \eqref{Mk}. We call \(n_k\) the \(k\)-th digit of \(n\).
For \(x=(x_k)_{k\ge0}\in\Gm\), define
\[
 r_k(x):=e^{2\pi ix_k/m_k},
 \qquad
 \psi_n(x):=\prod_{k\ge0}r_k(x)^{n_k}.
\]
The family \((\psi_n)_{n\ge0}\) is the associated Vilenkin
character system.

For an \(X\)-valued Vilenkin polynomial \(f\) and a set
\(A\subseteq\N\), define
\[
 \widehat f(n)
 :=
 \int_{\Gm}f(x)\overline{\psi_n(x)}\,d\mu(x),
 \qquad
 P_Af
 :=
 \sum_{n\in A}\widehat f(n)\psi_n.
\]
The character identities
\[
 \psi_a\psi_b=\psi_{a\oplus b},
 \qquad
 \overline{\psi_a}=\psi_{0\ominus a}
\]
yield the modulation identity
\begin{equation}\label{eq:modulation-identity}
 P_{c\oplus A}f
 =
 \psi_cP_A(\overline{\psi_c}f),
 \qquad
 c\in\N,\quad A\subseteq\N.
\end{equation}

The canonical filtration on \(\Gm\) is given by
\[
 \cF_k:=\sigma(x_0,\ldots,x_{k-1}),
 \qquad k\ge1,
\]
with \(\cF_0\) the trivial \(\sigma\)-algebra. Let \(E_k\)
denote conditional expectation onto \(\cF_k\). Then
\[
 E_k=P_{[0,M_k)}.
\]
For \(k\ge0\) and \(0\le\ell<m_k\), define the finer frequency
blocks and their associated projections by
\[
 \delta_{k,\ell}:=[\ell M_k,(\ell+1)M_k),
 \qquad
 \Delta_{k,\ell}:=P_{\delta_{k,\ell}}.
\]
Since
\(\delta_{k,\ell}=\ell M_k\oplus[0,M_k)\),
the modulation identity \eqref{eq:modulation-identity} gives
\begin{equation}\label{finer}
 \Delta_{k,\ell}f
 =
 \psi_{\ell M_k}E_k(\overline{\psi_{\ell M_k}}f).
\end{equation}
Consequently, the martingale-difference operators associated
with \((\cF_k)_{k\ge0}\) satisfy
\[
 d_k:=E_{k+1}-E_k
 =
 \sum_{\ell=1}^{m_k-1}\Delta_{k,\ell},
 \qquad k\ge0.
\]

\subsection{Basic definitions for Banach spaces}

We refer to \cite{Pisier2016,HNVWII} for background on type
and the UMD property. Throughout, Banach spaces are complex
unless explicitly stated otherwise. All Gaussian variables
below are normalized complex Gaussian variables, and
\(\varepsilon=(\varepsilon_j)_j\) denotes an independent
Rademacher sequence.

\begin{definition}[Type~\(2\)]
A Banach space \(X\) has \emph{type~\(2\)} if there exists
a constant \(C<\infty\) such that, for every finite
family \(x_1,\ldots,x_n\in X\),
\[
 \left(
  \mathbb E_\varepsilon
  \norm{\sum_{j=1}^n\varepsilon_jx_j}_X^2
 \right)^{1/2}
 \le
 C
\left(\sum_{j=1}^n\norm{x_j}_X^2\right)^{1/2}.
\]
We denote the least admissible constant by \(t_{2,X}\).
\end{definition}

\begin{remark}\label{rgeq}
The type~\(2\) property can equivalently be defined using
Gaussian random variables; see
\cite[Theorem~12.26]{DiestelJarchowTonge1995}.
More precisely, \(X\) has type~\(2\) if and only if there
exists a constant \(C'<\infty\) such that, for every finite
family \(x_1,\ldots,x_n\in X\),
\[
 \left(
  \mathbb E_\gamma
  \norm{\sum_{j=1}^n\gamma_jx_j}_X^2
 \right)^{1/2}
 \le
 C'\left(\sum_{j=1}^n\norm{x_j}_X^2\right)^{1/2},
\]
where \((\gamma_j)_j\) is an independent sequence of normalized
complex Gaussian variables, normalized by \(\mathbb E|\gamma_j|^2=1\).
Taking \(G_{2,X}\) to be the least admissible constant, we have
$$\frac 1 {\sqrt \pi} \,t_{2,X}\le G_{2,X}\le t_{2,X}.$$
\end{remark}

\begin{definition}[UMD spaces]
 A Banach space \(X\) has the
\emph{UMD property} if for some $1<r<\infty$ (equivalently for every
$1<r<\infty$), there exists \(C_r<\infty\) such that
\[
 \norm{\sum_{k=1}^n\theta_kd_k}_{L^r(\Omega;X)}
 \le
 C_r
 \norm{\sum_{k=1}^nd_k}_{L^r(\Omega;X)}
\]
for every finite \(X\)-valued martingale difference sequence
\((d_k)_{k=1}^n\) in \(L^r(\Omega;X)\) on any filtered
probability space, and every choice of signs
\(\theta_k\in\{-1,1\}\). For each fixed \(r\), we denote the least
admissible constant by \(\beta_{r,X}\).
\end{definition}
% \begin{definition}[Rademacher space]
% For a Banach space \(X\), the \emph{Rademacher space}
% \(\operatorname{Rad}(X)\) is the closure in
% \(L^2(\Omega_\varepsilon;X)\) of all finite sums
% \[
%  \sum_{j=1}^n\varepsilon_jx_j,
%  \qquad n\ge1,\quad x_1,\ldots,x_n\in X.
% \]
% It is equipped with the inherited norm, so that
% \[
%  \norm{\sum_{j=1}^n\varepsilon_jx_j}_{\operatorname{Rad}(X)}
%  :=
%  \left(
%   \mathbb E_\varepsilon
%   \norm{\sum_{j=1}^n\varepsilon_jx_j}_X^2
%  \right)^{1/2}.
% \]
% \end{definition}

\subsection{Noncommutative \texorpdfstring{\(L^p\)}{Lp}-spaces
and column and row sequence spaces}

Let \((\cM,\tau)\) be a semifinite von Neumann algebra
equipped with a normal faithful semifinite trace.
For \(1\le p<\infty\), let \(L^p(\cM)\) denote the associated
noncommutative \(L^p\)-space. For a finite sequence
\((x_s)_s\subset L^p(\cM)\), define the column and row norms by
\begin{align*}
 \norm{(x_s)_s}_{L^p(\cM;\ell^2_c)}
 &:=
 \norm{\left(\sum_sx_s^*x_s\right)^{1/2}}_{L^p(\cM)},\\
 \norm{(x_s)_s}_{L^p(\cM;\ell^2_r)}
 &:=
 \norm{\left(\sum_sx_sx_s^*\right)^{1/2}}_{L^p(\cM)}.
\end{align*}
The spaces \(L^p(\cM;\ell^2_c)\) and \(L^p(\cM;\ell^2_r)\)
are the completions of the finitely supported sequences
under these norms.

For \(1<p<\infty\), let \(p'\) be the conjugate exponent,
so that \(1/p+1/p'=1\). With respect to the trace pairing
\(\sum_s\tau(x_s^*y_s)\), these spaces satisfy the isometric
dualities
\begin{equation}\label{eq:column-row-duality}
 \bigl(L^p(\cM;\ell^2_c)\bigr)^*
 =
 L^{p'}(\cM;\ell^2_c),
 \qquad
 \bigl(L^p(\cM;\ell^2_r)\bigr)^*
 =
 L^{p'}(\cM;\ell^2_r).
\end{equation}
We refer to Pisier--Xu~\cite{PisierXu1997} for further
background on noncommutative \(L^p\)-spaces and the
column and row sequence spaces.

\section{Necessity of the UMD property and the type
\texorpdfstring{\(2\)}{2} condition}\label{s3}

In this section, we prove Theorem~\ref{thm:walsh-necessity}.  Our
argument follows the approach of Hyt\"onen, Torrea, and Yakubovich
\cite[Section~2]{HTY}.
Throughout this section, the generating sequence \(\mathbf m\) is fixed,
and constants denoted by \(C_{p,X,\mathbf m}\) may increase from line to line.
 
 \subsection{The \texorpdfstring{\(\operatorname{LPR}_{p,=}\)}
{LPRp equal} case}
 We first treat the
$\operatorname{LPR}_{p,=}$ property. For the UMD property,
applying the $\operatorname{LPR}_{p,=}$ inequality to a family consisting of
a single interval $I=[0,n)$ shows that the Vilenkin partial-sum
operators are uniformly bounded on $L^p(\Gm;X)$. It follows from
\cite[Corollary 1.3]{CH2026} that $X$ is UMD.

For the type~$2$ condition, we first prove the following elementary consequence of the equal-length LPR estimate, which is the main ingredient in the proof of the type~$2$ condition.

\begin{lemma}\label{lem:two-parameter-walsh}
Assume that $X$ has the $\operatorname{LPR}_{p,=}$ property.  Let
$f_1,\ldots,f_J$ be $X$-valued Vilenkin polynomials satisfying
\[
 \operatorname{supp}\widehat f_j\subset[0,M_r),
 \qquad 1\le j\le J
\]
for some $r\in\mathbb N.$
Then
\begin{equation}\label{eq:two-parameter-walsh}
 \left(
  \mathbb E_\varepsilon
  \left\|
   \sum_{j=1}^J
   \sum_{k=0}^{M_r-1}
   \varepsilon_{j,k}\widehat f_j(k)
  \right\|_X^p
 \right)^{1/p}
 \le
 C_{p,X,\mathbf m}
 \left\|
  \max_{|\delta_j|=1}
  \left\|
   \sum_{j=1}^J\delta_jf_j(\,\cdot\,)
  \right\|_X
 \right\|_{L^p(\Gm)},
\end{equation}
where $(\varepsilon_{j,k})_{j,k}$ is an independent Rademacher
family.
\end{lemma}

\begin{proof}
For $1\le j\le J$, set
\[
 a_j:=jM_r,
 \qquad
 F:=\sum_{j=1}^J\psi_{a_j}f_j.
\]
Since the first $r$ digits of $a_j$ vanish, we have
\[
 a_j\oplus k=a_j+k,
 \qquad 0\le k<M_r.
\]
Consequently, the singleton intervals
\[
 \bigl\{\{a_j+k\}:1\le j\le J,\ 0\le k<M_r\bigr\}
 =
 \bigl\{\{a_j\oplus k\}:1\le j\le J,\ 0\le k<M_r\bigr\}
\]
are pairwise disjoint.  Moreover,
\[
 P_{\{a_j+k\}}F
 =
 \widehat f_j(k)\psi_{a_j+k}.
\]
Applying the $\operatorname{LPR}_{p,=}$ property to this family gives
\[
 \left(
  \mathbb E_\varepsilon
  \left\|
   \sum_{j=1}^J
   \sum_{k=0}^{M_r-1}
   \varepsilon_{j,k}
   \widehat f_j(k)\psi_{a_j+k}
  \right\|_{L^p(\Gm;X)}^p
 \right)^{1/p}
 \le
  C_{p,X,\mathbf m}\norm{F}_{L^p(\Gm;X)}.
\]
For each fixed $x\in\Gm$, all the scalars
$\psi_{a_j+k}(x)$ are unimodular.  Hence the complex contraction
principle shows that the expression on the left-hand side is equivalent,
up to a universal constant, to
\[
 \left(
  \mathbb E_\varepsilon
  \left\|
   \sum_{j=1}^J
   \sum_{k=0}^{M_r-1}
   \varepsilon_{j,k}\widehat f_j(k)
  \right\|_X^p
 \right)^{1/p}.
\]
On the other hand,
\[
 F(x)=\sum_{j=1}^J\psi_{a_j}(x)f_j(x),
\]
and therefore
\[
 \norm{F(x)}_X
 \le
 \max_{|\delta_j|=1}
 \left\|
  \sum_{j=1}^J\delta_jf_j(x)
 \right\|_X.
\]
Combining these estimates proves
\eqref{eq:two-parameter-walsh}.
\end{proof}

Now we can show that $X$ has type~$2.$
Fix $r\ge1$ and set
\(
 N:=M_r.
\)
Let
\[
 F_r=\{A_1,\ldots,A_N\}
\]
be the family of atoms of $\mathcal F_r$.  Thus
\[
 \mu(A_j)=\frac1N,
 \qquad 1\le j\le N.
\]
For each $1\le j\le N$, set
\(
 \phi_j:=\1_{A_j}.
\)
Since $\phi_j$ is $\mathcal F_r$-measurable, it is a Vilenkin
polynomial whose Fourier-Vilenkin support is contained in $[0,N)$.
Moreover, every $\psi_k$, $0\le k<N$, is constant on $A_j$.
Consequently,
\begin{equation}\label{eq:atom-walsh-coefficients}
 \widehat\phi_j(k)
 =
 \int_{A_j}\overline{\psi_k(x)}\,d\mu(x)
 =
 \frac{\sigma_{j,k}}{N},
 \qquad
 |\sigma_{j,k}|=1,
 \qquad
 0\le k<N.
\end{equation}

Let $x_1,\ldots,x_N\in X$ satisfy
\[
 \norm{x_j}_X\le1,
 \qquad 1\le j\le N.
\]
Applying Lemma~\ref{lem:two-parameter-walsh} to
\(
 f_j:=\phi_jx_j
\)
and using \eqref{eq:atom-walsh-coefficients}, we may absorb the
unimodular factors $\sigma_{j,k}$ by the complex contraction principle.
It follows that
\begin{align}
 &\left(
  \mathbb E_\varepsilon
  \left\|
   \frac1N
   \sum_{j=1}^N
   \sum_{k=0}^{N-1}
   \varepsilon_{j,k}x_j
  \right\|_X^p
 \right)^{1/p}
 \notag\\
 &\qquad\le
  C_{p,X,\mathbf m}
 \left\|
  \max_{|\delta_j|=1}
  \left\|
   \sum_{j=1}^N
   \delta_j\phi_j(\,\cdot\,)x_j
  \right\|_X
 \right\|_{L^p(\Gm)}
 \le
 C_{p,X,\mathbf m}.
 \label{eq:double-rademacher-bound}
\end{align}
Here the last inequality follows from the pairwise disjointness of the
sets $A_1,\ldots,A_N$ and the assumption
$\norm{x_j}_X\le1$.
For $1\le j\le N$, define
\[
 \xi_j
 :=
 \frac1N\sum_{k=0}^{N-1}\varepsilon_{j,k}.
\]
The random variables $\xi_1,\ldots,\xi_N$ are independent and
symmetric.  We
may therefore realize
\(
 \xi_j=\eta_j\rho_j
\)
in distribution,
where $(\eta_j)_{j=1}^N$ is an independent Rademacher family,
$\rho_j$ has the same distribution as $|\xi_j|$, and the two families
$(\eta_j)_j$ and $(\rho_j)_j$ are independent (after a harmless enlargement of the underlying probability space if necessary; see \cite[Example~4.4.13]{HNVWI}).  The scalar Khintchine
inequality therefore gives
\[
 \mathbb E_\rho\rho_j
 =
 \mathbb E_\varepsilon|\xi_j|
 =
 \frac1N
 \mathbb E_\varepsilon
 \left|
  \sum_{k=0}^{N-1}\varepsilon_k
 \right|
 \ge
 \frac{c}{\sqrt N},
\]
where $c>0$ is an absolute constant.

For each fixed realization of $(\eta_j)_{j=1}^N$, the function
\[
 (t_1,\ldots,t_N)
 \longmapsto
 \left\|
  \sum_{j=1}^N\eta_jt_jx_j
 \right\|_X^p
\]
is convex on $[0,\infty)^N$.  Jensen's inequality, followed by
 taking $p$-th roots, therefore yields
\begin{align}
 \left(
  \mathbb E_\varepsilon
  \left\|
   \sum_{j=1}^N\xi_jx_j
  \right\|_X^p
 \right)^{1/p}
 &=
 \left(
  \mathbb E_\eta\mathbb E_\rho
  \left\|
   \sum_{j=1}^N\eta_j\rho_jx_j
  \right\|_X^p
 \right)^{1/p}
 \notag\\
 &\ge
 \frac{c}{\sqrt N}
 \left(
  \mathbb E_\eta
  \left\|
   \sum_{j=1}^N\eta_jx_j
  \right\|_X^p
 \right)^{1/p}.
 \label{eq:double-to-single-rad}
\end{align}
Combining \eqref{eq:double-rademacher-bound} and
\eqref{eq:double-to-single-rad}, we obtain
\begin{equation}\label{eq:equal-norm-type2-p}
 \left(
  \mathbb E_\eta
  \left\|
   \sum_{j=1}^N\eta_jx_j
  \right\|_X^p
 \right)^{1/p}
 \le
 C'_{p,X,\mathbf m}\sqrt N.
\end{equation}
Since $p\ge2$, this also implies
\[
 \left(
  \mathbb E_\eta
  \left\|
   \sum_{j=1}^N\eta_jx_j
  \right\|_X^2
 \right)^{1/2}
 \le
 C'_{p,X,\mathbf m}\sqrt N.
\]

We now pass from lengths of the form $M_r$ to an arbitrary length.
Let $n\ge1$ and let $x_1,\ldots,x_n\in X$ satisfy
$\norm{x_j}_X\le1$.  Choose $r$ so large that
\(
 N:=M_r>n
\)
and set
\(
 q:=\left\lfloor N/n\right\rfloor.
\)
Then
\(
 n\le N/q\le2n.
\)
In \eqref{eq:equal-norm-type2-p}, repeat each vector $x_j$ exactly
$q$ times and fill the remaining $N-nq$ positions with $x_j=0$.  We obtain
\[
 \left(
  \mathbb E_\eta
  \left\|
   \sum_{j=1}^n
   \sum_{\ell=1}^q
   \eta_{j,\ell}x_j
  \right\|_X^p
 \right)^{1/p}
 \lesssim_{p,X,\mathbf m}
 \sqrt N,
\]
where $(\eta_{j,\ell})_{j,\ell}$ is an i.i.d.\ Rademacher family.

For $1\le j\le n$, put
\[
 \zeta_j:=\sum_{\ell=1}^q\eta_{j,\ell}.
\]
The random variables $\zeta_1,\ldots,\zeta_n$ are independent and
symmetric, and the scalar Khintchine inequality gives
\[
 \mathbb E_\eta|\zeta_j|
 =
 \mathbb E_\eta
 \left|
  \sum_{\ell=1}^q\eta_{j,\ell}
 \right|
 \gtrsim
 \sqrt q.
\]
We may again realize
\(
 \zeta_j=\eta_j\rho_j,
\)
where $(\eta_j)_{j=1}^n$ is an independent Rademacher family,
$\rho_j$ has the same distribution as $|\zeta_j|$, and
$(\eta_j)_j$ and $(\rho_j)_j$ are independent.  The same Jensen
argument as in \eqref{eq:double-to-single-rad}, followed by taking
$p$-th roots, yields
\[
 \left(
  \mathbb E_\zeta
  \left\|
   \sum_{j=1}^n\zeta_jx_j
  \right\|_X^p
 \right)^{1/p}
 \gtrsim
 \sqrt q
 \left(
  \mathbb E_\eta
  \left\|
   \sum_{j=1}^n\eta_jx_j
  \right\|_X^p
 \right)^{1/p}.
\]
Consequently,
\[
 \left(
  \mathbb E_\eta
  \left\|
   \sum_{j=1}^n\eta_jx_j
  \right\|_X^p
 \right)^{1/p}
 \lesssim_{p,X,\mathbf m}
 \sqrt{\frac{N}{q}}
 \lesssim_{p,X,\mathbf m}
 \sqrt n.
\]
This is the equal-norm type~$2$ estimate on the unit ball of $X$.
By the equal-norm theorem of James \cite{James1978}, it is
equivalent to type~$2$ in the usual sense.  This proves the necessity
of the UMD property and the type~$2$ condition under the
$\operatorname{LPR}_{p,=}$ assumption.

\subsection{The \texorpdfstring{\(\operatorname{LPR}_{p,=}^{\oplus}\)}
{LPRp equal plus} case}

It remains to consider the $\operatorname{LPR}_{p,=}^{\oplus}$ property.  The
argument is essentially identical.  Indeed, $[0,n)$ is both a canonical
interval and a Vilenkin interval. Hence, applying the $\operatorname{LPR}_{p,=}^{\oplus}$ property to a family consisting of a single interval yields the uniform boundedness of the Vilenkin partial-sum operators. It then follows from \cite[Corollary~1.3]{CH2026} that $X$ is UMD.

For the type~$2$ condition, apply the
$\operatorname{LPR}_{p,=}^{\oplus}$ property to the pairwise disjoint
equal-length Vilenkin intervals
\[
 \{a_j\oplus k\},
 \qquad
 1\le j\le J,
 \quad
 0\le k<M_r.
\]
This gives Lemma~\ref{lem:two-parameter-walsh} with the same proof.
The remainder of the type~$2$ argument applies without change.  This
completes the proof of Theorem~\ref{thm:walsh-necessity}.

\section{Proof of Theorem~\ref{thm:counterexample-main}}\label{s4}

In this section, we prove Theorem~\ref{thm:counterexample-main}
by constructing a counterexample on Schatten classes.
Specifically, we encode the \(\{0,1\}\)-patterns derived from
Sylvester--Hadamard matrices into the mixed-radix digits of
the endpoints of pairwise disjoint canonical intervals of equal
length. We then construct matrix-valued Vilenkin polynomials
whose interval projections realize these patterns up to
unitary factors. The singular-value structure of the Hadamard
matrices allows us to show that the ratios of the output
Schatten norms to the corresponding input column or row norms
diverge as the matrix dimension tends to infinity.

\begin{proof}[Proof of Theorem~\ref{thm:counterexample-main}]
It suffices to show that, for each $\dagger\in\{c,r\}$, the operator
\[
 (P_{I_s})_s
 :
 L^p(\cN_{\mathbf m})
 \longrightarrow
 L^p(\cN_{\mathbf m};\ell^2_\dagger),
 \qquad
 f\longmapsto (P_{I_s}f)_s,
\]
fails to be uniformly bounded over finite families of pairwise disjoint
equal-length intervals $(I_s)_s$ on $\mathbb N$.
We first consider the column case.  By
\eqref{eq:column-row-duality},
\[
 \bigl(L^p(\cN_{\mathbf m};\ell^2_c)\bigr)^*
 =
 L^{p'}(\cN_{\mathbf m};\ell^2_c).
\]
Since each $P_{I_s}$ is self-adjoint with respect to the
canonical dual pairing, the column estimate is equivalent to
\begin{equation}\label{eq:dual-column-interval-estimate}
 \norm{
  \sum_sP_{I_s}u_s
 }_{L^{p'}(\cN_{\mathbf m})}
 \le
 C_{p,\mathbf m}
 \norm{(u_s)_s}_{L^{p'}(\cN_{\mathbf m};\ell^2_c)}
\end{equation}
for every finite sequence $(u_s)_s\subset L^{p'}(\cN_{\mathbf m})$, with a
constant independent of the interval family.

\medskip
\noindent
\textit{Step 1: Construction of the equal-length intervals.}
Fix
\(
 N:=2^d, d\ge1.
\)
Let $H_N$ be the Sylvester--Hadamard matrix of order $N$, so that
\[
 H_NH_N^*=NI_N,
 \qquad
 (H_N)_{k,s}\in\{-1,1\}.
\]
Let $J_N$ be the $N\times N$ all-ones matrix and define
\[
 E
 :=
 \frac{J_N+H_N}{2}
 =
 (\varepsilon_{k,s})_{
  \substack{0\le k<N\\1\le s\le N}},
 \qquad
 \varepsilon_{k,s}\in\{0,1\}.
\]
For $1\le s\le N$, set
\[
 B_s
 :=
 \sum_{k=0}^{N-1}\varepsilon_{k,s}M_k.
\]
Since $\varepsilon_{k,s}\in\{0,1\}$ and $m_k\ge2$, we have
\(
 0\le B_s<M_N.
\)
Now define
\[
 b_s:=sM_{N+1}+B_s,
 \qquad
 I_s:=[b_s-M_N,b_s).
\]
Clearly,
\[
 |I_s|=M_N,
 \qquad 1\le s\le N.
\]
Moreover,
\[
 (b_{s+1}-M_N)-b_s
 =
 (m_N-1)M_N+B_{s+1}-B_s
 \ge1,
\]
because $0\le B_s\le M_N-1$.  Thus the intervals
$I_1,\ldots,I_N$ are pairwise disjoint.

\medskip
\noindent
\textit{Step 2: Construction of the column test sequence.}
Let
\[
 q_N:=\sum_{j=1}^Ne_{jj}\in B(\ell^2).
\]
We identify the finite-dimensional corner
$q_NB(\ell^2)q_N$ with $\mathbb M_N$.  All the functions constructed
below take values in this corner.
Define
\[
 U(x)
 :=
 \sum_{k=0}^{N-1}
 \overline{r_k(x)}e_{k+1,k+1}.
\]
For each $1\le s\le N$, set
\[
 h_s(x)
 :=
 \sum_{k=0}^{N-1}
 \overline{r_k(x)}e_{s,k+1}
 =
 e_{ss}J_NU(x),
 \qquad
 u_s(x):=\psi_{b_s}(x)h_s(x).
\]

We next compute $P_{I_s}u_s$ explicitly.  Since
\(
 b_s=sM_{N+1}+B_s,
\)
the first $N$ digits of $b_s$ are precisely
\(
 \varepsilon_{0,s},\ldots,\varepsilon_{N-1,s}.
\)
Hence, for every $0\le k<N$,
\[
 b_s\ominus M_k
 =
 \begin{cases}
  b_s-M_k,
    & \varepsilon_{k,s}=1,\\[1mm]
  b_s+(m_k-1)M_k,
    & \varepsilon_{k,s}=0.
 \end{cases}
\]
If $\varepsilon_{k,s}=1$, then
\[
 b_s\ominus M_k
 =
 b_s-M_k
 \in[b_s-M_N,b_s)
 =
 I_s.
\]
If $\varepsilon_{k,s}=0$, then
\(
 b_s\ominus M_k>b_s,
\)
and therefore $b_s\ominus M_k\notin I_s$.  It follows that
\begin{equation}\label{eq:exact-column-projection}
 P_{I_s}u_s
 =
 \psi_{b_s}
 \sum_{k=0}^{N-1}
 \varepsilon_{k,s}\overline{r_k}\,e_{s,k+1}.
\end{equation}

\medskip
\noindent
\textit{Step 3: Realization of the Hadamard mask.}
Define
\[
 R(x)
 :=
 \sum_{s=1}^N\psi_{b_s}(x)e_{ss}.
\]
Both $R(x)$ and $U(x)$ are unitary elements of $\mathbb M_N$.  By
\eqref{eq:exact-column-projection},
\begin{align*}
 \sum_{s=1}^NP_{I_s}u_s(x)
 &=
 \sum_{s=1}^N
 \sum_{k=0}^{N-1}
 \varepsilon_{k,s}
 \psi_{b_s}(x)
 \overline{r_k(x)}e_{s,k+1}
 \\
 &=
 R(x)E^{\mathsf T}U(x).
\end{align*}
The unitary invariance of the Schatten norm therefore gives
\begin{equation}\label{eq:column-output-Spprime}
 \norm{
  \sum_{s=1}^NP_{I_s}u_s
 }_{L^{p'}(\cN_{\mathbf m})}
 =
 \norm{E^{\mathsf T}}_{S^{p'}_N}
 =
 \norm{E}_{S^{p'}_N}.
\end{equation}

Since
\[
 E=\frac{J_N+H_N}{2},
\]
the triangle inequality yields
\begin{align}
 \norm{E}_{S^{p'}_N}
 &\ge
 \frac12
 \left(
  \norm{H_N}_{S^{p'}_N}
  -
  \norm{J_N}_{S^{p'}_N}
 \right)
 \notag\\
 &=
 \frac12
 \left(
  N^{1/p'+1/2}-N
 \right).
 \label{eq:column-Spprime-lower-hadamard}
\end{align}
Indeed, all $N$ singular values of $H_N$ are equal to $\sqrt N$,
whereas $J_N$ has rank one and its unique nonzero singular value is
equal to $N$.

\medskip
\noindent
\textit{Step 4: Computation of the column input norm.}
Since $\psi_{b_s}$ is scalar-valued and unimodular, we have
\(
 u_s^*u_s=h_s^*h_s.
\)
Using $J_N^2=NJ_N$, we obtain pointwise
\begin{align*}
 \sum_{s=1}^Nu_s^*u_s
 &=
 U^*
 \left(
  \sum_{s=1}^NJ_Ne_{ss}J_N
 \right)
 U
 \\
 &=
 U^*J_N^2U
 =
 NU^*J_NU.
\end{align*}
Since $J_N\ge0$ and $J_N^2=NJ_N$, it follows that
\[
 \left(
  \sum_{s=1}^Nu_s^*u_s
 \right)^{1/2}
 =
 U^*J_NU.
\]
Consequently,
\begin{equation}\label{eq:column-input-norm}
 \norm{(u_s)_{s=1}^N}_{L^{p'}(\cN_{\mathbf m};\ell^2_c)}
 =
 \norm{J_N}_{S^{p'}_N}
 =
 N.
\end{equation}

\medskip
\noindent
\textit{Step 5: Failure of the column estimate.}
Combining \eqref{eq:column-output-Spprime},
\eqref{eq:column-Spprime-lower-hadamard}, and
\eqref{eq:column-input-norm}, we obtain
\begin{equation}\label{eq:dual-column-ratio-counterexample}
 \frac{
  \norm{
   \sum_{s=1}^NP_{I_s}u_s
  }_{L^{p'}(\cN_{\mathbf m})}
 }{
  \norm{(u_s)_{s=1}^N}_{L^{p'}(\cN_{\mathbf m};\ell^2_c)}
 }
 \ge
 \frac12
 \left(
  N^{1/p'-1/2}-1
 \right).
\end{equation}
Since
\[
 \frac1{p'}-\frac12
 =
 \frac12-\frac1p
 >
 0,
\]
the right-hand side of
\eqref{eq:dual-column-ratio-counterexample} tends to infinity as
$N\to\infty$.  Thus the dual column estimate
\eqref{eq:dual-column-interval-estimate} cannot hold uniformly.
By duality, the column square-function estimate in
\eqref{eq:false-operator-rubio} fails.

\medskip
\noindent
\textit{Step 6: Failure of the row estimate.}
For the row case, define
\[
 h_s(x)
 :=
 \sum_{k=0}^{N-1}
 \overline{r_k(x)}e_{k+1,s}
 =
 U(x)J_Ne_{ss},
 \qquad
 u_s(x):=\psi_{b_s}(x)h_s(x).
\]
The same frequency computation as in Step~2 gives
\[
 P_{I_s}u_s
 =
 \psi_{b_s}
 \sum_{k=0}^{N-1}
 \varepsilon_{k,s}\overline{r_k}\,e_{k+1,s}.
\]
Consequently,
\[
 \sum_{s=1}^NP_{I_s}u_s
 =
 UER,
\]
and hence
\[
 \norm{
  \sum_{s=1}^NP_{I_s}u_s
 }_{L^{p'}(\cN_{\mathbf m})}
 =
 \norm{E}_{S^{p'}_N}.
\]

On the other hand,
\begin{align*}
 \sum_{s=1}^Nu_su_s^*
 &=
 U
 \left(
  \sum_{s=1}^NJ_Ne_{ss}J_N
 \right)
 U^*
 \\
 &=
 UJ_N^2U^*
 =
 NUJ_NU^*.
\end{align*}
Therefore,
\[
 \left(
  \sum_{s=1}^Nu_su_s^*
 \right)^{1/2}
 =
 UJ_NU^*,
\]
and thus
\[
 \norm{(u_s)_{s=1}^N}_{L^{p'}(\cN_{\mathbf m};\ell^2_r)}
 =
 \norm{J_N}_{S^{p'}_N}
 =
 N.
\]
It follows from \eqref{eq:column-Spprime-lower-hadamard} that
\[
 \frac{
  \norm{
   \sum_{s=1}^NP_{I_s}u_s
  }_{L^{p'}(\cN_{\mathbf m})}
 }{
  \norm{(u_s)_{s=1}^N}_{L^{p'}(\cN_{\mathbf m};\ell^2_r)}
 }
 \ge
 \frac12
 \left(
  N^{1/p'-1/2}-1
 \right).
\]
This quantity tends to infinity as $N\to\infty$.  Hence the row
estimate in \eqref{eq:false-operator-rubio-row} also fails.  This completes
the proof.
\end{proof}

\section{Proof of Theorem~\ref{thm:oplus-equal-umd-type2-main}}\label{s5}
This section proves
Theorem~\ref{thm:oplus-equal-umd-type2-main}.  
The proof proceeds in several steps. We first derive a randomized
Fourier-coefficient estimate from the type~\(2\) property and use it to
deduce the LPR inequality for equal-scale cosets in Vilenkin systems. We then establish, via a transference argument, a
uniform two-parameter randomization estimate for equal-length Fourier
projections on finite cyclic groups.
Next,  we embed
the Vilenkin intervals into digital rectangular envelopes and deduce the LPR inequality for rectangular envelopes.
Finally, this reduces the 
\(\operatorname{LPR}_{p,=}^{\oplus}\) estimate to the uniform boundedness
of Vilenkin partial-sum operators on
\(L^p(\Gm;X)\), which yields the theorem by \cite[Theorem 1.2]{CH2026}.

Throughout this section, we set
\(
 \Gamma_k:=[0,M_k),k\ge0.
\)
Then we have $E_k=P_{\Gamma_k}.$

\subsection{A randomized Fourier-coefficient estimate}

We shall use the following standard consequence of the type-\(2\) property, which is similar to \cite[Lemma~3.2]{HTY} but gives a quantitative bound in terms of $p$ and $t_{2,X}$.
We include the proof because the argument applies to arbitrary orthonormal
systems on probability spaces, rather than only to the trigonometric
system.

\begin{lemma}\label{lem:type2-randomized-coefficients}
Let \((H,\nu)\) be a probability space, let
\((\chi_j)_{j\in J}\) be a finite orthonormal family of scalar-valued
functions in \(L^2(H)\), and suppose that \(X\) is a Banach space with type \(2\). Then, for
every \(2\le p<\infty\) and every \(g\in L^p(H;X)\),
\begin{equation*}
\left(
\mathbb E_\varepsilon
\left\|
\sum_{j\in J}\varepsilon_j
\int_H g(t)\overline{\chi_j(t)}\,d\nu(t)
\right\|_X^p
\right)^{1/p}
\lesssim_{p}
t_{2,X}\|g\|_{L^2(H;X)}
\le
t_{2,X}\|g\|_{L^p(H;X)}.
\end{equation*}
\end{lemma}

\begin{proof}
By the Kahane--Khintchine inequalities, it suffices to prove
\[
\left(
\mathbb E_\varepsilon
\left\|
\sum_{j\in J}\varepsilon_j
\int_H g(t)\overline{\chi_j(t)}\,d\nu(t)
\right\|_X^2
\right)^{1/2}
\lesssim_{p}
t_{2,X}\|g\|_{L^2(H;X)}.
\]
By a density argument, we may assume that
\[
g=\sum_{\ell=1}^N\mathbf1_{A_\ell}x_\ell,
\]
where the sets \(A_\ell\) are pairwise disjoint and have positive measure.
Set
\[
y_\ell:=\nu(A_\ell)^{1/2}x_\ell,
\qquad
u_{j,\ell}
:=
\nu(A_\ell)^{-1/2}
\int_{A_\ell}\overline{\chi_j(t)}\,d\nu(t).
\]
By Bessel's inequality, the operator
\[
U=(u_{j,\ell})_{j\in J,\,1\le\ell\le N},
\qquad
Ue_\ell=\sum_{j\in J}u_{j,\ell}e_j,
\]
is a contraction from \(\ell^2_N\) to \(\ell^2(J)\). Let
\[
T:\ell^2_N\to X,
\qquad
Te_\ell=y_\ell,
\]
and let \((\gamma_j)_{j\in J}\) and
\((\gamma_\ell')_{1\le\ell\le N}\) be two independent families of
normalized complex Gaussian random variables. The ideal property of Gaussian norms
\cite[Theorem~9.1.10]{HNVWII}, applied to \(T\) and \(U^{\mathsf T}\), gives
\begin{align*}
\left(
\mathbb E_\gamma
\left\|
\sum_{j\in J}\gamma_j
\int_Hg(t)\overline{\chi_j(t)}\,d\nu(t)
\right\|_X^2
\right)^{1/2}
&=
\left(
\mathbb E_\gamma
\left\|
\sum_{\ell=1}^N
\left(\sum_{j\in J}\gamma_ju_{j,\ell}\right)y_\ell
\right\|_X^2
\right)^{1/2}
\\
&=
\left(
\mathbb E_\gamma
\left\|
\sum_{j\in J}\gamma_jTU^{\mathsf T}e_j
\right\|_X^2
\right)^{1/2}
\\
&\le
\|U^{\mathsf T}\|_{\ell^2(J)\to\ell^2_N}
\left(
\mathbb E_{\gamma'}
\left\|
\sum_{\ell=1}^N\gamma_\ell'Te_\ell
\right\|_X^2
\right)^{1/2}
\\
&\le
\left(
\mathbb E_{\gamma'}
\left\|
\sum_{\ell=1}^N\gamma_\ell'y_\ell
\right\|_X^2
\right)^{1/2}.
\end{align*}
The Gaussian formulation of the type-\(2\) inequality in
Remark~\ref{rgeq} therefore yields
\[
\left(
\mathbb E_\gamma
\left\|
\sum_{j\in J}\gamma_j
\int_Hg(t)\overline{\chi_j(t)}\,d\nu(t)
\right\|_X^2
\right)^{1/2}
\le t_{2,X}
\left(\sum_{\ell=1}^N\|y_\ell\|_X^2\right)^{1/2}
=
t_{2,X}\|g\|_{L^2(H;X)}.
\]
By the standard Gaussian comparison inequality
\cite[Proposition~12.11]{DiestelJarchowTonge1995}, 
\begin{align*}
\left(
\mathbb E_\varepsilon
\left\|
\sum_{j\in J}\varepsilon_j
\int_Hg(t)\overline{\chi_j(t)}\,d\nu(t)
\right\|_X^2
\right)^{1/2}
\le
\sqrt{\pi}
\left(
\mathbb E_\gamma
\left\|
\sum_{j\in J}\gamma_j
\int_Hg(t)\overline{\chi_j(t)}\,d\nu(t)
\right\|_X^2
\right)^{1/2}.
\end{align*}
This completes the proof.
\end{proof}

The following lemma yields the LPR inequality
for equal-scale cosets.

\begin{lemma}\label{lem:oplus-coset-randomization}
Let \(X\) be a Banach space with type \(2\), \(2\le p<\infty\)  and 
\((D_\kappa)_{\kappa\in\Lambda}\) be a finite family of distinct cosets
\(
D_\kappa=\kappa\oplus\Gamma_k,
\)
where $\Lambda\subset\mathbb N$ is the collection of $\kappa$ such that the first \(k\) digits of every \(\kappa\) vanish. Then for any $f\in L^p(\Gm;X),$
\begin{equation}\label{eq:oplus-coset-randomization}
\left(
\mathbb E_{\varepsilon}
\left\|
\sum_{\kappa\in\Lambda}
\varepsilon_\kappa P_{D_\kappa}f
\right\|_{L^p(\Gm;X)}^p
\right)^{1/p}
\lesssim_{p}t_{2,X}
\|f\|_{L^p(\Gm;X)}.
\end{equation}
\end{lemma}

\begin{proof}
For $k\in\mathbb N,$ write
\[
\Gm=G_{<k}\times G_{\ge k},
\qquad
G_{<k}:=\prod_{j<k}\mathbb Z_{m_j},
\qquad
G_{\ge k}:=\prod_{j\ge k}\mathbb Z_{m_j}.
\]
For \(u\in G_{<k}\) and \(\kappa\in\Lambda\), define
\[
a_\kappa(u)
:=
\int_{G_{\ge k}}
f(u,v)\overline{\psi_\kappa(u,v)}\,d\mu(v).
\]
Since the first \(k\) digits of \(\kappa\) vanish,
\(\psi_\kappa:(u,v)\mapsto \psi_\kappa(u,v)\) depends only on the coordinate $v$. Moreover, by the property of the conditional expectation $E_k,$
\begin{align*}
E_kf(u,v)
=
\int_{G_{\ge k}}f(u,w)\,d\mu(w).
\end{align*}
The modulation identity \eqref{eq:modulation-identity} therefore gives
\begin{equation*}
P_{D_\kappa}f(u,v)
=
\psi_\kappa(v)a_\kappa(u).
\end{equation*}

For fixed \(u\), applying
Lemma~\ref{lem:type2-randomized-coefficients} to the finite orthonormal
family \((\psi_\kappa)_{\kappa\in\Lambda}\) in \(L^2(G_{\ge k})\), together
with the complex contraction principle, gives
\begin{align*}
\left(
\mathbb E_\varepsilon
\left\|
\sum_{\kappa\in\Lambda}
\varepsilon_\kappa\psi_\kappa(\cdot)a_\kappa(u)
\right\|_{L^p(G_{\ge k};X)}^p
\right)^{1/p}
&\lesssim_p
\left(
\mathbb E_\varepsilon
\left\|
\sum_{\kappa\in\Lambda}
\varepsilon_\kappa a_\kappa(u)
\right\|_X^p
\right)^{1/p}
\\
&\lesssim_{p}t_{2,X}
\|f(u,\cdot)\|_{L^2(G_{\ge k};X)}\\
&\le
t_{2,X}\|f(u,\cdot)\|_{L^p(G_{\ge k};X)}.
\end{align*}
Taking the \(p\)-th powers, integrating in \(u\), and then taking the
\(p\)-th root proves \eqref{eq:oplus-coset-randomization}.
\end{proof}

\subsection{A two-parameter estimate on finite cyclic groups}
We now use the sampling factorization to establish a two-parameter randomization estimate for group
Fourier projections on \(\mathbb Z_m\). We adopt the convention in \cite[Section 4, Section 5]{CH2026}.
For $m\ge2$, let
\begin{equation}\label{eq:centered-residue-system}
 R_m
 :=
 \left\{-\left\lfloor\frac m2\right\rfloor,
 \ldots,\left\lceil\frac m2\right\rceil-1\right\}
\end{equation}
be the centered complete residue system, and let
$\rho_m:\mathbb Z_m\to R_m$ be the centered representative map given by \[
\rho_m(q)
=
\begin{cases}
q,
& 0\le q\le \left\lceil \dfrac{m}{2}\right\rceil-1,\\[6pt]
q-m,
& \left\lceil \dfrac{m}{2}\right\rceil\le q\le m-1.
\end{cases}
\] For an
interval $J\subset R_m$, define the centered cyclic Fourier projection $P_{m,J}$ by
\begin{equation}\label{centered pro}
 P_{m,J}h(x)
 :=
 \sum_{q\in J}\widehat h(q\bmod m)e^{2\pi iqx/m},
 \qquad x\in\mathbb Z_m.
\end{equation}

Consider the canonical embedding
\[
\vartheta_m:\mathbb Z_m\to\mathbb T,
\qquad
a\mapsto e^{2\pi ia/m}.
\]
We identify \(\mathbb Z_m\) with its image under this embedding. For
\(a\in\mathbb Z_m\), let
\((a\mid a+1\bmod m)\) denote the open arc from
\(\vartheta_m(a)\) to \(\vartheta_m(a+1\bmod m)\) in the cyclic order
\[
0\longrightarrow1\longrightarrow\cdots\longrightarrow m-1
\longrightarrow0.
\] See Figure \ref{fig:cyclic-order-Zm}.

\begin{figure}[htbp]
    \centering
    \includegraphics[width=0.5\textwidth]
    {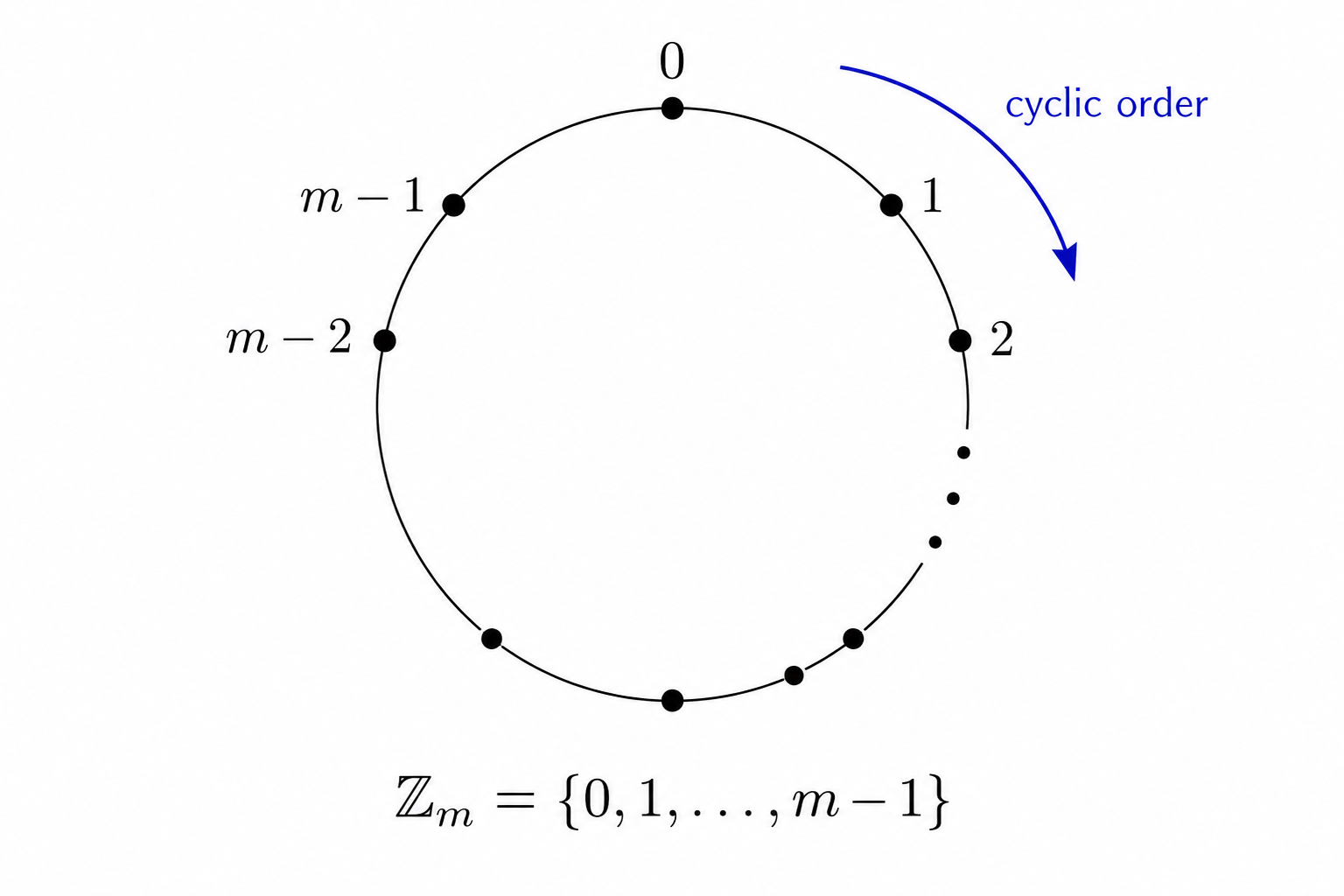}
    \caption{The cyclic ordering of the elements of \(\mathbb Z_m\).}
    \label{fig:cyclic-order-Zm}
\end{figure}

For an interval
\[
J=\{a,\ldots,a+L\bmod m\}\subsetneq\mathbb Z_m,
\]
define
\[
J^\circ
:=
\bigcup_{j=0}^{L-1}
(a+j\bmod m\mid a+j+1\bmod m).
\]
Let \(P_J^{\mathbb Z_m}\) denote the group Fourier projection on
\(\mathbb Z_m\) associated with the interval \(J\):
\[
P_J^{\mathbb Z_m}h(x)
:=
\sum_{q\in J}
\widehat h(q)e^{2\pi iqx/m},
\qquad x\in\mathbb Z_m.
\]

\begin{lemma}\label{lem:two-parameter-cyclic-equal}
Let \(X\) be a $\operatorname{UMD}$ Banach space of type \(2\), and let \(2\le p<\infty\). For
every \(m\ge2\), let \((h_\rho)_\rho\) be a finite family of \(X\)-valued
functions on \(\mathbb Z_m\). For each fixed \(\rho\), let
\((J_{\rho,j})_j\) be a pairwise disjoint family of intervals in
\(\mathbb Z_m\), and assume that all the intervals \(J_{\rho,j}\), over
both indices \(\rho\) and \(j\), have the same cardinality. Then
\begin{align}\label{eq:two-parameter-cyclic-equal}
\left(
\mathbb E_{\varepsilon,\delta}\left\|
\sum_{\rho,j}
\delta_\rho\varepsilon_{\rho,j}
P_{J_{\rho,j}}^{\mathbb Z_m}h_\rho
\right\|_{
L^p(\mathbb Z_m;X)}^p
\right)^{1/p}
\lesssim_{p}t_{2,X}\beta_{p,X}^2
\left(
\mathbb E_{\delta}\left\|
\sum_\rho\delta_\rho h_\rho
\right\|_{L^p(\mathbb Z_m;X)}^p
\right)^{1/p}.
\end{align}
Here \((\delta_\rho)_\rho\) and
\((\varepsilon_{\rho,j})_{\rho,j}\) are mutually independent i.i.d.\ Rademacher families.
\end{lemma}

We shall use  the
equal-length LPR inequality on $\mathbb T$, established by Hyt{\"o}nen
et al.~\cite[Theorem~1.1]{HTY}. Although the explicit dependence of the
$\operatorname{LPR}_{p,=}^{\mathbb T}$ constant is not stated in their
article, it can be extracted from their proof. We record the resulting
quantitative estimate in the following lemma.

\begin{lemma}
\label{Hytonen}
Let \(X\) be a $\operatorname{UMD}$ Banach space with type~\(2\), and let
\(2\le p<\infty\). For every finite family \((I_j)_j\)
of pairwise disjoint intervals in \(\mathbb Z\) of equal
cardinality and every \(f\in L^p(\mathbb T;X)\),
\begin{equation}\label{eq:torus-equal-lpr-quantitative}
 \left(
\mathbb E_\varepsilon\left\|
  \sum_j\varepsilon_jP_{I_j}^{\mathbb T}f
 \right\|_{L^p(\mathbb T;X)}^p
\right)^{1/p}
 \lesssim_p
 t_{2,X}\beta_{p,X}^2
 \|f\|_{L^p(\mathbb T;X)}.
\end{equation}
\end{lemma}
\begin{proof}[Sketch of proof]
We only provide a sketch of the proof.
We track the dependence on the geometric constants of \(X\)
in the argument of \cite[Section~3]{HTY}.
We first refine the dual estimate used in
\cite[Lemma~3.2]{HTY}.
For a finite set \(J\subset\mathbb Z\) and a family
\(\lambda=(\lambda_j)_{j\in J}\subset X^*\), define
\[
 \|\lambda\|_*
 :=
 \sup\left\{
  \left|\sum_{j\in J}\langle\lambda_j,x_j\rangle\right|:
  x_j\in X,\quad
  \left\|\sum_{j\in J}\varepsilon_jx_j
  \right\|_{\operatorname{Rad}(X)}\le1
 \right\}.
\]
The definition of \(\operatorname{Rad}(X)\) can be found in \cite[Proof of Lemma~3.1]{HTY}.
Set
\[
 g_\lambda(t)
 :=
 \sum_{j\in J}e^{-2\pi ijt}\lambda_j.
\]
For every \(\varphi\in L^2(\mathbb T;X)\),
Lemma~\ref{lem:type2-randomized-coefficients}, applied with
exponent \(2\), gives
\begin{align*}
 \left|
  \int_{\mathbb T}
  \langle g_\lambda(t),\varphi(t)\rangle\,dt
 \right|
 &=
 \left|
  \sum_{j\in J}
  \langle\lambda_j,\widehat\varphi(j)\rangle
 \right|
\le
 \|\lambda\|_*
 \left(
  \mathbb E_\varepsilon
  \left\|
   \sum_{j\in J}\varepsilon_j\widehat\varphi(j)
  \right\|_X^2
 \right)^{1/2}
 \\
 &\lesssim
 t_{2,X}\|\lambda\|_*
 \|\varphi\|_{L^2(\mathbb T;X)}.
\end{align*}
By \(L^2\)-duality, it follows that
\begin{equation}\label{eq:hty-refined-dual-estimate}
 \|g_\lambda\|_{L^2(\mathbb T;X^*)}
 \lesssim
t_{2,X}\|\lambda\|_*.
\end{equation}

Choose the de la Vall\'ee Poussin multipliers \(m_j\)
according to the common length of intervals as in \cite[Section~3]{HTY}. 
Using the exact duality formula above and
\eqref{eq:hty-refined-dual-estimate} in the kernel argument
of \cite[Section~3]{HTY}, we obtain
\[
 \left\|
  \sum_{j\in J}\varepsilon_jT_{m_j}^{\mathbb T}f
 \right\|_{L^p(\mathbb T;\operatorname{Rad}(X))}
 \lesssim_p
t_{2,X}\|f\|_{L^p(\mathbb T;X)}.
\]
We omit the details since they are the same as in \cite{HTY}.
% For \(p>2\), this follows from the boundedness of the
% Hardy--Littlewood maximal operator on \(L^{p/2}(\mathbb T)\).
% For \(p=2\), one retains the intermediate averaging estimates
% before passing to the maximal function and uses the
% \(L^1\)-boundedness of the normalized averaging operators.
The Kahane--Khintchine inequalities therefore yield
\[
 \left(
  \mathbb E_\varepsilon
  \left\|
   \sum_{j\in J}\varepsilon_jT_{m_j}^{\mathbb T}f
  \right\|_{L^p(\mathbb T;X)}^p
 \right)^{1/p}
 \lesssim_p
 t_{2,X}\|f\|_{L^p(\mathbb T;X)}.
\]

Finally, applying
\cite[Lemma~4.3]{CH2026} (a quantitative version of \cite[Lemma 3.1]{HTY}) to
\(T_{m_j}^{\mathbb T}f\), we obtain
\begin{align*}
 \left(
  \mathbb E_\varepsilon
  \left\|
   \sum_{j\in J}\varepsilon_jP_{I_j}^{\mathbb T}f
  \right\|_{L^p(\mathbb T;X)}^p
 \right)^{1/p}
 &\lesssim_p
 \beta_{p,X}^2
 \left(
  \mathbb E_\varepsilon
  \left\|
   \sum_{j\in J}\varepsilon_jT_{m_j}^{\mathbb T}f
  \right\|_{L^p(\mathbb T;X)}^p
 \right)^{1/p}
 \\
 &\lesssim_p
t_{2,X}\beta_{p,X}^2
 \|f\|_{L^p(\mathbb T;X)}.
\end{align*}
This proves \eqref{eq:torus-equal-lpr-quantitative}.
\end{proof}

% The second one is a standard result for vector-valued bounded-variation multipliers on the torus. We give a short proof for completeness.

% \begin{lemma}\label{Mar}
% Given
% \(b:\mathbb Z\to\mathbb C\) supported in $R_m$ with bounded vartiation $$\operatorname{Var}(b):=\sum_{n\in\mathbb Z}|b(n+1)-b(n)|<\infty,$$ for any  \(f\in L^p(\mathbb T;X)\), we have
% \begin{equation}\label{eq:periodic-marcinkiewicz-quantitative}
%  \|T_b^{\mathbb T}f\|_{L^p(\mathbb T;X)}
%  \lesssim
%  \operatorname{Var}(b)
%  \|f\|_{L^p(\mathbb T;X)}.
% \end{equation}
% \end{lemma}
% \begin{proof}
    
% \end{proof}

\begin{proof}[Proof of Lemma \ref{lem:two-parameter-cyclic-equal}]

\medskip

\noindent
\textit{Step 1: Reduction from \(\mathbb Z_m\) to \(\mathbb Z\).}

We first reduce arbitrary cyclic intervals on \(\mathbb Z_m\) to intervals in
\(\mathbb Z\) contained in the centered residue system \(R_m\).
For \(a\in\mathbb Z_m\), define the pointwise multiplier
\[
(\operatorname{M}_a h)(x)
:=
e^{2\pi iax/m}h(x),
\qquad x\in\mathbb Z_m.
\]
Fix \(\rho\). We may assume that all intervals \(J_{\rho,j}\) are proper
subsets of \(\mathbb Z_m\). Indeed, if one of them equals
\(\mathbb Z_m\), then the equal-cardinality and pairwise-disjointness
assumptions imply that the corresponding family consists of the whole
group, and the desired estimate is immediate from
\(P_{\mathbb Z_m}^{\mathbb Z_m}=\operatorname{Id}\).

Let $(a\mid a+1\bmod m)$ be the open arc defined above. Choose \(a\in\mathbb Z_m\) such that
\[
(a\mid a+1\bmod m)\cap J_{\rho,j}^\circ=\varnothing,
\qquad\text{for every }j.
\]
 Such a choice exists. Indeed, if
\(
\bigcup_jJ_{\rho,j}\neq\mathbb Z_m,
\)
we choose the cut in a gap of the complement. If \(
\bigcup_jJ_{\rho,j}=\mathbb Z_m,
\), then, since they are pairwise disjoint, we may choose \(a\in\mathbb Z_m\) and distinct indices
\(j_-\) and \(j_+\) such that
\(
 a\in J_{\rho,j_-}, a+1\in J_{\rho,j_+}.
\)
Then the arc determined by $a$ is a valid place for the cut. These two
configurations are illustrated in Figure~\ref{fig:cyclic-cut-two-cases}.

\begin{figure}[htbp]
    \centering
    \includegraphics[width=0.7\textwidth]
    {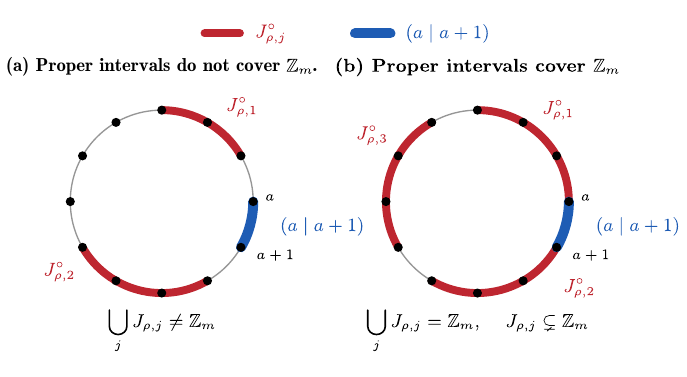}
    \caption{The two possible configurations of the family
    \((J_{\rho,j})_j\).}
    \label{fig:cyclic-cut-two-cases}
\end{figure}

Once such a point \(a\) has been chosen, set
\[
a_\rho
\equiv
\left\lceil\frac m2\right\rceil-1-a
\pmod m,
\qquad
a_\rho\in\mathbb Z_m.
\]
Then
\[
\rho_m(a+a_\rho)
=
\left\lceil\frac m2\right\rceil-1,
\qquad
\rho_m(a+1+a_\rho)
=
-\left\lfloor\frac m2\right\rfloor.
\]
Hence
\[
\widetilde J_{\rho,j}
:=
\rho_m(J_{\rho,j}+a_\rho)
\subset R_m
\]
are intervals in \(\mathbb Z\) for any \(j\). In particular,
\[
\widetilde J_{\rho,j}
=
\{b_{\rho,j},b_{\rho,j}+1,\ldots,b_{\rho,j}+L_{\rho,j}\}
\]
for suitable \(b_{\rho,j}\in\mathbb Z\) and
\(L_{\rho,j}\in\mathbb N\). The family
\((\widetilde J_{\rho,j})_j\) remains pairwise disjoint and has the same
cardinalities as the original cyclic family.
With our Fourier convention,
\[
\widehat{\operatorname{M}_{a_\rho}h}(q)
=
\widehat h(q-a_\rho),
\qquad q\in\mathbb Z_m.
\]
Therefore, by the definition of the centered cyclic Fourier projection in
\eqref{centered pro},
\begin{equation}\label{eq:cyclic-arc-modulation}
P_{m,\widetilde J_{\rho,j}}\operatorname{M}_{a_\rho}
=
\operatorname{M}_{a_\rho}P_{J_{\rho,j}}^{\mathbb Z_m}.
\end{equation}

We next establish the two randomized comparisons that will be used below.
For fixed \((\varepsilon,x)\), set
\[
F_\rho(\varepsilon,x)
:=
\sum_j\varepsilon_{\rho,j}
P_{J_{\rho,j}}^{\mathbb Z_m}h_\rho(x),
\qquad
\lambda_\rho(x):=e^{2\pi ia_\rho x/m}.
\]
Since \(|\lambda_\rho(x)|=1\), the complex contraction principle gives
\[
\left(
\mathbb E_\delta
\left\|
\sum_\rho\delta_\rho F_\rho(\varepsilon,x)
\right\|_X^p
\right)^{1/p}
\lesssim
\left(
\mathbb E_\delta
\left\|
\sum_\rho\delta_\rho
\lambda_\rho(x)F_\rho(\varepsilon,x)
\right\|_X^p
\right)^{1/p}.
\]
Taking \(p\)-th powers, integrating in \(\varepsilon\) and \(x\),
and then taking \(p\)-th roots and using
\eqref{eq:cyclic-arc-modulation}, we obtain
\begin{align}
&\left(
\mathbb E_{\varepsilon,\delta}\left\|
\sum_{\rho,j}\delta_\rho\varepsilon_{\rho,j}
P_{J_{\rho,j}}^{\mathbb Z_m}h_\rho
\right\|_{L^p(\mathbb Z_m;X)}^p
\right)^{1/p}\nonumber\\
&\qquad\lesssim
\left(
\mathbb E_{\varepsilon,\delta}\left\|
\sum_{\rho,j}\delta_\rho\varepsilon_{\rho,j}
P_{m,\widetilde J_{\rho,j}}
\operatorname{M}_{a_\rho}h_\rho
\right\|_{L^p(\mathbb Z_m;X)}^p
\right)^{1/p}.
\label{eq:reduce-cyclic-arcs-to-centered}
\end{align}
Similarly, for every fixed \(x\in\mathbb Z_m\), the complex contraction
principle gives
\[
\left(
\mathbb E_\delta
\left\|
\sum_\rho\delta_\rho\lambda_\rho(x)h_\rho(x)
\right\|_X^p
\right)^{1/p}
\lesssim
\left(
\mathbb E_\delta
\left\|
\sum_\rho\delta_\rho h_\rho(x)
\right\|_X^p
\right)^{1/p}.
\]
Taking \(p\)-th powers, integrating in \(x\), and then taking
\(p\)-th roots, we obtain
\begin{equation}\label{eq:remove-cyclic-row-modulations}
\left(
\mathbb E_\delta\left\|
\sum_\rho\delta_\rho
\operatorname{M}_{a_\rho}h_\rho
\right\|_{L^p(\mathbb Z_m;X)}^p
\right)^{1/p}
\lesssim
\left(
\mathbb E_\delta\left\|
\sum_\rho\delta_\rho h_\rho
\right\|_{L^p(\mathbb Z_m;X)}^p
\right)^{1/p}.
\end{equation}
It therefore suffices to prove
\begin{align}\label{reduction}
&\left(
\mathbb E_{\varepsilon,\delta}\left\|
\sum_{\rho,j}\delta_\rho\varepsilon_{\rho,j}
P_{m,\widetilde J_{\rho,j}}
\operatorname{M}_{a_\rho}h_\rho
\right\|_{L^p(\mathbb Z_m;X)}^p
\right)^{1/p}\\
&\qquad\lesssim_p t_{2,X}\beta_{p,X}^2
\left(
\mathbb E_{\delta}\left\|
\sum_\rho\delta_\rho
\operatorname{M}_{a_\rho}h_\rho
\right\|_{L^p(\mathbb Z_m;X)}^p
\right)^{1/p}.
\end{align}

\medskip
\noindent
\textit{Step 2: The corresponding estimate on \(\mathbb T\).}

Let \((g_\rho)_\rho\) be a finite family of \(X\)-valued trigonometric
polynomials on \(\mathbb T\). The general case follows by density. Choose
a finite interval \(K_\rho\subset\mathbb Z\) containing both
\[
\operatorname{supp}\widehat g_\rho
\quad\text{and}\quad
\bigcup_j\widetilde J_{\rho,j},
\]
and choose integers \(N_\rho\) such that the intervals
\((N_\rho+K_\rho)_\rho\) are pairwise disjoint. Define
\[
G_\delta(t)
:=
\sum_\rho\delta_\rho e^{2\pi iN_\rho t}g_\rho(t).
\]
The choice of the intervals \(K_\rho\) ensures that the intervals
\(N_\rho+\widetilde J_{\rho,j}\) are pairwise disjoint over all
\((\rho,j)\), and that
\(\widehat g_\sigma\) vanishes on
\(-N_\sigma+N_\rho+\widetilde J_{\rho,j}\) whenever \(\rho\ne\sigma\).
Consequently,
\begin{equation}\label{eq:translated-row-projection}
P_{N_\rho+\widetilde J_{\rho,j}}^{\mathbb T}G_\delta
=
\delta_\rho e^{2\pi iN_\rho t}
P_{\widetilde J_{\rho,j}}^{\mathbb T}g_\rho.
\end{equation}

For each fixed realization of \((\delta_\rho)_\rho\), we may apply
 Lemma \ref{Hytonen}, since
\((N_\rho+\widetilde J_{\rho,j})_{\rho,j}\) is a pairwise disjoint family
of intervals in \(\mathbb Z\) with equal cardinality. Hence, by the
\(\operatorname{LPR}_{p,=}^{\mathbb T}\) property of \(X\), we have
\begin{align*}
\left(
\mathbb E_{\varepsilon}\left\|
\sum_{\rho,j}\varepsilon_{\rho,j}
P_{N_\rho+\widetilde J_{\rho,j}}^{\mathbb T}G_\delta
\right\|_{L^p(\mathbb T;X)}^p
\right)^{1/p}
\lesssim_{p}t_{2,X}\beta_{p,X}^2
\|G_\delta\|_{L^p(\mathbb T;X)}.
\end{align*}
Taking the \(L^p\)-norm of \(\delta\) and using
\eqref{eq:translated-row-projection}, this becomes
\begin{align}
&\left(
\mathbb E_{\varepsilon,\delta}\left\|
\sum_{\rho,j}\delta_\rho\varepsilon_{\rho,j}
e^{2\pi iN_\rho t}
P_{\widetilde J_{\rho,j}}^{\mathbb T}g_\rho
\right\|_{L^p(\mathbb T;X)}^p
\right)^{1/p}\nonumber\\
&\qquad\lesssim_{p}t_{2,X}\beta_{p,X}^2
\left(
\mathbb E_{\delta}\left\|
\sum_\rho\delta_\rho e^{2\pi iN_\rho t}g_\rho
\right\|_{L^p(\mathbb T;X)}^p
\right)^{1/p}.
\label{eq:modulated-torus-two-parameter}
\end{align}

For fixed \((\varepsilon,t)\), apply the complex contraction principle in
the \(\delta_\rho\)-variables with coefficients
\(e^{-2\pi iN_\rho t}\) on the left-hand side. For fixed \(t\), apply the
same principle with coefficients \(e^{2\pi iN_\rho t}\) on the right-hand
side. We then obtain
\begin{align}
\left(
\mathbb E_{\varepsilon,\delta}\left\|
\sum_{\rho,j}\delta_\rho\varepsilon_{\rho,j}
P_{\widetilde J_{\rho,j}}^{\mathbb T}g_\rho
\right\|_{L^p(\mathbb T;X)}^p
\right)^{1/p}
\lesssim_{p}t_{2,X}\beta_{p,X}^2
\left(
\mathbb E_{\delta}\left\|
\sum_\rho\delta_\rho g_\rho
\right\|_{L^p(\mathbb T;X)}^p
\right)^{1/p}.
\label{eq:torus-two-parameter-final}
\end{align}

\medskip
\noindent
\textit{Step 3: Transfer from \(\mathbb T\) to \(\mathbb Z_m\) via a sampling factorization.}

For any \(m\ge2\) and interval $J\subset R_m$, define the finitely supported torus symbol
\begin{equation*}
b_{m,J}(n)
:=
\begin{cases}
\Phi(n/m),&n\in J,\\
0,&n\notin J,
\end{cases}
\qquad n\in\mathbb Z.
\end{equation*}
where
\[
\Phi(u)=\frac{\pi u}{\sin(\pi u)},
\qquad |u|\le\frac12,
\qquad
\Phi(0)=1.
\]
Fix a function $\eta\in C_c^\infty(\mathbb R)$ supported in $(-4/3,4/3)$
such that $\eta=\Phi$ on $[-1/2,1/2]$. Then
$ b_{m,J}(n)=\mathbf 1_J(n)\eta(n/m)$ for every $n\in\mathbb Z$,
which implies the identity
\begin{align}\label{smooth}
    T_{b_{m,J}}^{\mathbb T}
=
P_J^{\mathbb T}T_{\eta(\cdot/m)}^{\mathbb T}.
\end{align}
Here $T_{\eta(\cdot/m)}^{\mathbb T}$ denotes the Fourier multiplier on the torus
with symbol $\eta(\cdot/m)|_{\mathbb Z}$. The de Leeuw transference theorem therefore gives, for every $1\le p\le\infty$ and
every Banach space $Y$,
\begin{equation}\label{eq:cm-uniform-multiplier}
    \sup_{m\ge 2}\left\|T_{\eta(\cdot/m)}^{\mathbb T}\right\|_{L^p(\mathbb T;Y)\to L^p(\mathbb T;Y)}\lesssim 1.
\end{equation}

We now use the cyclic sampling construction from
\cite[Proposition~4.4 and Lemma~5.1]{CH2026}.
Let $D_m$ be the convolution operator defined in  \cite[Proposition~4.4]{CH2026}, and $\mathcal J_m,\mathcal U_m,\mathcal E_m$ be the contractive maps defined in \cite[Section 5]{CH2026}. In particular, \(\mathcal J_m:L^p(\mathbb Z_m;X)\to \mathcal E_mL^p(\mathbb T;X)\) is an
isometry and $\mathcal U_m:\mathcal E_mL^p(\mathbb T;X)\to L^p(\mathbb Z_m;X)$ is the inverse of $\mathcal J_m$. Hence for any Banach space $Y$ we have
\begin{equation}\label{eq:Am-uniform-bound}
    \sup_{m\ge2}
\|\mathcal J_mD_m\mathcal U_m\mathcal E_m\|_{L^p(\mathbb T;Y)\to L^p(\mathbb T;Y)}
\lesssim_{p}1.
\end{equation}
Using the sampling factorization from \cite[Lemma~5.1]{CH2026}, i.e.
$$\mathcal J_mP_{m,J}
=
\mathcal J_mD_m\mathcal U_m\mathcal E_m
T_{b_{m,J}}^{\mathbb T}\mathcal J_m,$$
and \eqref{smooth}, we obtain
\begin{equation}\label{eq:centered-projection-factorization}
\mathcal J_mP_{m,J}
=
\mathcal J_mD_m\mathcal U_m\mathcal E_mP_J^{\mathbb T}
T_{\eta(\cdot/m)}^{\mathbb T}\mathcal J_m,
\qquad J\subset R_m.
\end{equation}
Applying the isometric property of \(\mathcal J_m\), estimate
\eqref{eq:Am-uniform-bound} with
\(Y=L^p(\Omega_\delta\times\Omega_\varepsilon;X)\), and then
\eqref{eq:torus-two-parameter-final} with
\[
g_\rho
=
T_{\eta(\cdot/m)}^{\mathbb T}\mathcal J_m\operatorname{M}_{a_{\rho}}h_\rho,
\]
we obtain
\begin{align}
&
\left(
\mathbb E_{\varepsilon,\delta}\left\|
\sum_{\rho,j}\delta_\rho\varepsilon_{\rho,j}
P_{m,\widetilde J_{\rho,j}}\operatorname{M}_{a_{\rho}}h_\rho
\right\|_{L^p(\mathbb Z_m;X)}^p
\right)^{1/p}
\nonumber\\
&\overset{
\eqref{eq:centered-projection-factorization},
\,\eqref{eq:Am-uniform-bound}
}{\lesssim}
\left(
\mathbb E_{\varepsilon,\delta}\left\|
\sum_{\rho,j}\delta_\rho\varepsilon_{\rho,j}
P_{\widetilde J_{\rho,j}}^{\mathbb T}
T_{\eta(\cdot/m)}^{\mathbb T}\mathcal J_m\operatorname{M}_{a_{\rho}}h_\rho
\right\|_{L^p(\mathbb T;X)}^p
\right)^{1/p}
\nonumber\\
&\ \ \ \overset{
\eqref{eq:torus-two-parameter-final}
}{\lesssim_{p}} t_{2,X}\beta_{p,X}^2
\left(
\mathbb E_{\delta}\left\|
\sum_\rho\delta_\rho
T_{\eta(\cdot/m)}^{\mathbb T}\mathcal J_m\operatorname{M}_{a_{\rho}}h_\rho
\right\|_{L^p(\mathbb T;X)}^p
\right)^{1/p}
\nonumber\\
&\ \ \ \overset{
\eqref{eq:cm-uniform-multiplier}
}{\lesssim_{p}}t_{2,X}\beta_{p,X}^2
\left(
\mathbb E_{\delta}\left\|
\sum_\rho\delta_\rho\mathcal J_m\operatorname{M}_{a_{\rho}}h_\rho
\right\|_{L^p(\mathbb T;X)}^p
\right)^{1/p}\nonumber\\
&\ \ \ =t_{2,X}\beta_{p,X}^2
\left(
\mathbb E_{\delta}\left\|
\sum_\rho\delta_\rho \operatorname{M}_{a_{\rho}}h_\rho
\right\|_{L^p(\mathbb Z_m;X)}^p
\right)^{1/p},
\label{eq:centered-cyclic-two-parameter}
\end{align}
This proves \eqref{reduction} and hence
\eqref{eq:two-parameter-cyclic-equal}. 
\end{proof}

\subsection{An embedding of Vilenkin intervals into digital rectangular envelopes}\label{s4.4}

We now exploit the special geometry of the disjoint sets
\(I_s^\oplus:=\zeta_s\oplus[0,L)\). We use the following notation: For \(L\ge1\), let \(k\ge0\) be the unique index such that
\(
M_k\le L<M_{k+1}.
\)
Write
\begin{equation*}
L=qM_k+r,
\qquad
1\le q<m_k,
\qquad
0\le r<M_k,
\end{equation*}
and write
\begin{equation*}
\zeta_s
=
\kappa_s\oplus c_sM_k\oplus h_s,
\qquad
h_s\in\Gamma_k,
\quad
c_s\in\mathbb Z_{m_k},
\end{equation*}
where the first \(k+1\) digits of \(\kappa_s\) vanish. Define the cyclic
intervals
\begin{equation*}
J_s:=\left(c_s+\{0,\ldots,q-1\}\right)\bmod{m_k},
\qquad
\widetilde J_s:=\left(c_s+\{0,\ldots,q\}\right)\bmod{m_k},
\end{equation*}
and define the digital rectangular envelope of $I_s^\oplus$ by
\begin{equation}\label{decomB}
B_s
:=
\kappa_s\oplus
\bigl(\widetilde J_sM_k\oplus\Gamma_k\bigr).
\end{equation}
Here, for any $J\subset\mathbb Z_{m_k}$, we use the notation
\begin{equation*}
 J M_k\oplus\Gamma_k
 :=
 \bigcup_{a\in J}\bigl(aM_k\oplus\Gamma_k\bigr).
\end{equation*}

\begin{lemma}\label{lem:digital-rectangular-envelopes}
Let \(I_s^\oplus:=\zeta_s\oplus[0,L),s\in\mathcal S\) be a family of disjoint Vilenkin intervals. With the notation above,  we have \(I_s^\oplus\subset B_s\). Moreover, the index set \(\mathcal S\) can be
partitioned into three classes
\[
\mathcal S
=
\mathcal S^{(1)}\dunion\mathcal S^{(2)}\dunion\mathcal S^{(3)}
\]
such that, for every \(\nu\in\{1,2,3\}\) and every fixed \(\kappa\), the
intervals
\[
\{\widetilde J_s:s\in\mathcal S^{(\nu)},\ \kappa_s=\kappa\}
\]
are pairwise disjoint.
\end{lemma}

\begin{proof}
The mixed-radix expansion gives the disjoint decomposition
\begin{equation*}
[0,L)
=
\left(
\bigdunion_{a=0}^{q-1}
(aM_k\oplus\Gamma_k)
\right)
\dunion
\bigl(qM_k\oplus[0,r)\bigr)
=:I_1\dunion I_2.
\end{equation*}
Since \(h_s\oplus\Gamma_k=\Gamma_k\),
\[
\zeta_s\oplus I_1
=
\kappa_s\oplus(J_sM_k\oplus\Gamma_k).
\]
On the other hand, since
\(h_s\oplus[0,r)\subset\Gamma_k\),
\[
\zeta_s\oplus I_2
\subset
\kappa_s\oplus
\bigl(\{((c_s+q)\bmod m_k)M_k\}\oplus\Gamma_k\bigr).
\]
This proves \(I_s^\oplus\subset B_s\).

For each \(\kappa\), set
\[
\mathcal S_\kappa
:=
\{s:\kappa_s=\kappa\}.
\]
Order the intervals \((J_s)_{s\in\mathcal S_\kappa}\) according to the
cyclic order on \(\mathbb Z_{m_k}\),
\[
0\longrightarrow1\longrightarrow\cdots\longrightarrow m_k-1
\longrightarrow0.
\]
For each \(s\), let \(s^-\) and \(s^+\) denote its predecessor and
successor in this ordering.
The intervals \((J_s)_{s\in\mathcal S_\kappa}\) are pairwise disjoint,
since the corresponding rectangles
\[
\kappa\oplus(J_sM_k\oplus\Gamma_k)
\subset I_s^\oplus
\]
are pairwise disjoint. Note that
\[
\widetilde J_s
=
J_s\cup\{(c_s+q)\bmod m_k\},
\]
hence the pairwise disjointness and equal cardinality of
\((J_s)_{s\in\mathcal S_\kappa}\) imply that
\[
\left\{
t\in\mathcal S_\kappa\setminus\{s\}:
\widetilde J_s\cap\widetilde J_t\neq\varnothing
\right\}
\subseteq
\{s^-,s^+\}.
\]
In particular,
\[
\sup_{s\in\mathcal S_\kappa}
\#\left\{
t\in\mathcal S_\kappa\setminus\{s\}:
\widetilde J_s\cap\widetilde J_t\neq\varnothing
\right\}
\le2.
\]
Therefore, there exists a map
\(
c_\kappa:\mathcal S_\kappa\longrightarrow\{1,2,3\}
\)
such that
\[
s\neq t,\quad
c_\kappa(s)=c_\kappa(t)
\quad\Longrightarrow\quad
\widetilde J_s\cap\widetilde J_t=\varnothing.
\]
Define
\[
\mathcal S_\kappa^{(a)}
:=
\{s:\kappa_s=\kappa,\ c_\kappa(s)=a\},
\qquad a\in\{1,2,3\}.
\]
Then
\[
\mathcal S_\kappa
=
\mathcal S_\kappa^{(1)}
\,\dot\cup\,
\mathcal S_\kappa^{(2)}
\,\dot\cup\,
\mathcal S_\kappa^{(3)},
\]
and, for every \(\kappa\) and \(a\in\{1,2,3\}\),
\[
\bigl(\widetilde J_s\bigr)_{
s\in\mathcal S_\kappa^{(a)}}
\ \text{is pairwise disjoint}.
\]
Finally, define
\[
\mathcal S^{(\nu)}
:=
\bigcup_{\kappa}\mathcal S_\kappa^{(\nu)},
\qquad
\nu=1,2,3.
\]
This completes the proof.
\end{proof}

With Lemma \ref{lem:digital-rectangular-envelopes}, we prove the LPR inequality for the rectangular envelopes.
\begin{lemma}\label{lem:rectangular-envelope-estimate}
Let \(X\) be a $\operatorname{UMD}$ Banach space with type \(2\), and let
\(2\le p<\infty\). Under the hypotheses and notation of
Lemma~\ref{lem:digital-rectangular-envelopes},
\begin{equation}\label{eq:rectangular-envelope-estimate}
\left(\mathbb E_{\varepsilon}
\left\|
\sum_s\varepsilon_sP_{B_s}f
\right\|_{L^p(\Gm;X)}^p
\right)^{1/p}
\lesssim_{p}t_{2,X}^2\beta_{p,X}^2
\|f\|_{L^p(\Gm;X)}.
\end{equation}
The implicit constant is independent of \(\mathbf m\), \(L\), and the
family.
\end{lemma}

\begin{proof}
Fix \(\nu\in\{1,2,3\}\). For every \(\kappa\) such that
\(\mathcal S_\kappa^{(\nu)}\neq\varnothing\), set
\[
D_\kappa:=\kappa\oplus\Gamma_{k+1},
\qquad
f_\kappa
:=
\overline{\psi_\kappa}P_{D_\kappa}f.
\]
Then
\(
\operatorname{supp}\widehat f_\kappa\subset\Gamma_{k+1}.
\)
For \(J\subset\mathbb Z_{m_k}\), write
\[
\Pi_{k,J}
:=
P_{JM_k\oplus\Gamma_k}.
\]
By the definition \eqref{decomB} of \(B_s\), if
\(s\in\mathcal S_\kappa^{(\nu)}\), then \(B_s\subset D_\kappa\). Hence,
by the modulation identity \eqref{eq:modulation-identity}, we have
\begin{equation}\label{eq:envelope-fibre-projection}
P_{B_s}f
=
\psi_\kappa\Pi_{k,\widetilde J_s}f_\kappa.
\end{equation}

Let \((\delta_\kappa)_\kappa\) and \((\eta_s)_s\) be independent Rademacher
families. The products
\((\delta_{\kappa_s}\eta_s)_s\) form an independent Rademacher family.
Since
\(\operatorname{supp}\widehat f_\kappa\subset\Gamma_{k+1}\), the operator
\(\Pi_{k,\widetilde J_s}\) acts only on the \(k\)-th digit. More precisely,
for
\[
(u,x_k,v)\in
\Gm
=
G_{<k}\times\mathbb Z_{m_k}\times G_{>k},
\]
we have
\[
\Pi_{k,\widetilde J_s}f_\kappa(u,x_k,v)
=
\left(
P_{\widetilde J_s}^{\mathbb Z_{m_k}}
f_\kappa(u,\cdot,v)
\right)(x_k).
\]
Therefore, for every fixed \((u,v)\), applying
\eqref{eq:two-parameter-cyclic-equal} fibrewise in the \(k\)-th coordinate,
with \(\rho=\kappa\) and
\(\{\widetilde J_s:s\in\mathcal S_\kappa^{(\nu)}\}\), gives
\begin{align*}
&
\left(\mathbb E_{\delta,\eta}
\left\|
\sum_\kappa\delta_\kappa
\sum_{s\in\mathcal S_\kappa^{(\nu)}}
\eta_s\Pi_{k,\widetilde J_s}
f_\kappa(u,\cdot,v)
\right\|_{L^p(\mathbb Z_{m_k};X)}^p
\right)^{1/p}
\\
&\qquad\lesssim_{p}t_{2,X}\beta_{p,X}^2
\left(\mathbb E_{\delta}
\left\|
\sum_\kappa\delta_\kappa
f_\kappa(u,\cdot,v)
\right\|_{L^p(\mathbb Z_{m_k};X)}^p
\right)^{1/p}.
\end{align*}
Integrating in \((u,v)\), we obtain
\begin{align}\label{fibre}
&
\left(\mathbb E_{\delta,\eta}
\left\|
\sum_\kappa\delta_\kappa
\sum_{s\in\mathcal S_\kappa^{(\nu)}}
\eta_s\Pi_{k,\widetilde J_s}f_\kappa
\right\|_{L^p(\Gm;X)}^p
\right)^{1/p}
\nonumber\\
&\qquad\lesssim_{p}t_{2,X}\beta_{p,X}^2
\left(\mathbb E_{\delta}
\left\|
\sum_\kappa\delta_\kappa f_\kappa
\right\|_{L^p(\Gm;X)}^p
\right)^{1/p}.
\end{align}
Combining \eqref{fibre} with
\eqref{eq:envelope-fibre-projection}, the complex contraction principle,
and \eqref{eq:oplus-coset-randomization}, we obtain
\begin{align*}
\left(\mathbb E_{\varepsilon}
\left\|
\sum_{s\in\mathcal S^{(\nu)}}
\varepsilon_sP_{B_s}f
\right\|_{L^p(\Gm;X)}^p
\right)^{1/p}
&\ =
\left(\mathbb E_{\delta,\eta}
\left\|
\sum_\kappa\delta_\kappa\psi_\kappa
\sum_{s\in\mathcal S_\kappa^{(\nu)}}
\eta_s\Pi_{k,\widetilde J_s}f_\kappa
\right\|_{L^p(\Gm;X)}^p
\right)^{1/p}
\\
&\ \lesssim_p
\left(\mathbb E_{\delta,\eta}
\left\|
\sum_\kappa\delta_\kappa
\sum_{s\in\mathcal S_\kappa^{(\nu)}}
\eta_s\Pi_{k,\widetilde J_s}f_\kappa
\right\|_{L^p(\Gm;X)}^p
\right)^{1/p}
\\
&\overset{\eqref{fibre}}{\lesssim_{p}}t_{2,X}\beta_{p,X}^2
\left(\mathbb E_{\delta}
\left\|
\sum_\kappa\delta_\kappa f_\kappa
\right\|_{L^p(\Gm;X)}^p
\right)^{1/p}
\\
&\ \lesssim_pt_{2,X}\beta_{p,X}^2
\left(\mathbb E_{\delta}
\left\|
\sum_\kappa\delta_\kappa P_{D_\kappa}f
\right\|_{L^p(\Gm;X)}^p
\right)^{1/p}
\\
&\,\overset{\eqref{eq:oplus-coset-randomization}}{\lesssim_{p}}t_{2,X}^2\beta_{p,X}^2
\|f\|_{L^p(\Gm;X)}.
\end{align*}
Summing over \(\nu=1,2,3\) proves
\eqref{eq:rectangular-envelope-estimate}.
\end{proof}

\subsection{Reduction to the uniform partial-sum estimate}
For a Banach space $X$, put
\begin{equation*}
 \mathfrak S^p(\Gm;X)
 :=
 \sup_{n\ge1}
 \left\|
  P_{[0,n)}
 \right\|_{
  L^p(\Gm;X)
  \to
  L^p(\Gm;X)
 }.
\end{equation*}
For UMD spaces $X$ with type~$2$, the following proposition reduces the
$\operatorname{LPR}_{p,=}^\oplus$ estimate to the uniform boundedness
$$\mathfrak S^p(\Gm;X)<\infty.$$

\begin{proposition}\label{prop:oplus-reduction}
Let \(2\le p<\infty\), and let \(X\) be a UMD Banach space with type \(2\).
If
\begin{equation}\label{eq:rad-partial-sums-assumption}
\mathfrak S^p(\Gm;X)<\infty,
\end{equation}
then \(X\) has the \(\operatorname{LPR}_{p,=}^{\oplus}\) property. More precisely, for
any pairwise disjoint family
\(I_s^\oplus:=\zeta_s\oplus[0,L)\), we have
\begin{equation*}
\left(\mathbb E_{\varepsilon}
\left\|
\sum_s\varepsilon_s
P_{I_s^\oplus}f
\right\|_{L^p(\Gm;X)}^p
\right)^{1/p}
\lesssim_{p}t_{2,X}^2\beta_{p,X}^2
\mathfrak S^p(\Gm;X)
\|f\|_{L^p(\Gm;X)}.
\end{equation*}
\end{proposition}

\begin{proof}
Let \(B_s\) be the rectangular envelopes defined in
Section~\ref{s4.4}. Set
\(
g_s:=P_{B_s}f.
\)
Recall that, in the proof of
Lemma~\ref{lem:digital-rectangular-envelopes},
\(\zeta_s\oplus [0,L)\subset B_s\). Consequently,
\[
P_{\zeta_s\oplus [0,L)}g_s
=
P_{\zeta_s\oplus [0,L)}f.
\]
The modulation identity \eqref{eq:modulation-identity} gives
\[
P_{\zeta_s\oplus [0,L)}g_s
=
\psi_{\zeta_s}
P_{[0,L)}
\bigl(\overline{\psi_{\zeta_s}}g_s\bigr).
\]
Set
\(
F:=\sum_s\varepsilon_s\overline{\psi_{\zeta_s}}g_s.
\)
For every fixed
realization of \(\varepsilon\), we have
\[
\left\|P_{[0,L)}F(\varepsilon,\cdot)\right\|_{L^p(\Gm;X)}
\le
\mathfrak S^p(\Gm;X)
\left\|F(\varepsilon,\cdot)\right\|_{L^p(\Gm;X)}.
\]
Taking the \(L^p\)-norm in $\varepsilon$ therefore gives
\[
\left\|P_{[0,L)}F\right\|_{L^p(\Omega_\varepsilon\times\Gm;X)}
\le
\mathfrak S^p(\Gm;X)
\left\|F\right\|_{L^p(\Omega_\varepsilon\times\Gm;X)}.
\]
Consequently, applying the complex contraction principle
pointwise on \(\Gm\), followed by Fubini's theorem, we obtain
\begin{align*}
\left(
\mathbb E_\varepsilon
\left\|
\sum_s\varepsilon_sP_{\zeta_s\oplus [0,L)}f
\right\|_{L^p(\Gm;X)}^p
\right)^{1/p}
&\lesssim
\left(
\mathbb E_\varepsilon
\left\|
P_{[0,L)}F
\right\|_{L^p(\Gm;X)}^p
\right)^{1/p}
\\
&\le
\mathfrak S^p(\Gm;X)
\left(
\mathbb E_\varepsilon
\left\|
F
\right\|_{L^p(\Gm;X)}^p
\right)^{1/p}
\\
&\lesssim
\mathfrak S^p(\Gm;X)
\left(
\mathbb E_\varepsilon
\left\|
\sum_s\varepsilon_sP_{B_s}f
\right\|_{L^p(\Gm;X)}^p
\right)^{1/p}.
\end{align*}
 By
Lemma~\ref{lem:rectangular-envelope-estimate}, it follows that
\[
\left(
\mathbb E_\varepsilon
\left\|
\sum_s\varepsilon_sP_{\zeta_s\oplus [0,L)}f
\right\|_{L^p(\Gm;X)}^p
\right)^{1/p}
\lesssim_p
t_{2,X}^2\beta_{p,X}^2
\mathfrak S^p(\Gm;X)
\|f\|_{L^p(\Gm;X)}.
\]
This completes the proof.
\end{proof}

\subsection{Proof of Theorem~\ref{thm:oplus-equal-umd-type2-main} and Corollary~\ref{cor:oplus-common-interval}}
We now prove the main results for the Vilenkin formulation of the LPR inequality.
We first prove Theorem~\ref{thm:oplus-equal-umd-type2-main}.

\begin{proof}[Proof of Theorem~\ref{thm:oplus-equal-umd-type2-main}]
Since \(X\) is UMD, the uniform partial-sum estimate
\cite[Theorem~1.2]{CH2026}, applied with exponent \(p\) and range space \(X\), gives
\[
\mathfrak S^p(\Gm;X)\lesssim_p\beta_{p,X}^4,
\]
uniformly over all generating sequences \(\mathbf m\). Combining this with
Proposition~\ref{prop:oplus-reduction} yields
the conclusion.
\end{proof}

Now we can deduce Corollary~\ref{cor:oplus-common-interval} from Theorem~\ref{thm:oplus-equal-umd-type2-main}.
\begin{proof}[Proof of Corollary~\ref{cor:oplus-common-interval}]
Put \(\ell:=L_2-L_1\). If \(\ell=1\), then
\[
\zeta_s\oplus[L_1,L_2)
=\{\zeta_s\oplus L_1\}
=(\zeta_s\oplus L_1)\oplus[0,1),
\]
so the result follows from Theorem~\ref{thm:oplus-equal-umd-type2-main}.
Suppose that \(\ell\ge2\), and let \(\alpha_j(n)\) denote the \(j\)-th
digit in the mixed-radix expansion of \(n\). Write
\[
L_1=\sum_{j\ge0}\alpha_j(L_1)M_j,
\qquad
L_2-1=\sum_{j\ge0}\alpha_j(L_2-1)M_j.
\]
Let \(k\) be the largest index for which
\[
\alpha_k(L_1)\ne\alpha_k(L_2-1).
\]
Then
\[
L_1=h+aM_k+r,
\qquad
L_2-1=h+bM_k+t,
\]
where
\[
M_{k+1}\mid h,
\qquad
0\le a<b<m_k,
\qquad 0\le r,t<M_k.
\]
Define
\[
w:=h+(a+1)M_k,
\qquad
\ell_-:=w-L_1=M_k-r,
\qquad
\ell_+:=L_2-w=(b-a-1)M_k+t+1.
\]
Then \(L_1<w<L_2\) and
\begin{equation}\label{eq:common-interval-two-piece}
[L_1,L_2)
=[L_1,w)\dunion[w,L_2)
=\bigl((w-1)\ominus[0,\ell_- )\bigr)
\dunion
\bigl(w\oplus[0,\ell_+ )\bigr).
\end{equation}
Consequently, we may write
\(
I_s^{\oplus}=I_s^-\dunion I_s^+,
\)
where
\[
I_s^-=(\zeta_s\oplus(w-1))\ominus[0,\ell_-),
\qquad
I_s^+=(\zeta_s\oplus w)\oplus[0,\ell_+).
\]

The family \((I_s^+)_s\) is a pairwise disjoint family of Vilenkin
intervals of the special form \(\eta_s\oplus[0,\ell_+)\), where
\(\eta_s=\zeta_s\oplus w\). Theorem~\ref{thm:oplus-equal-umd-type2-main}
therefore gives
\[
\left(\mathbb E_{\varepsilon}
\left\|
 \sum_s\varepsilon_sP_{I_s^+}f
\right\|_{L^p(\Gm;X)}^p
\right)^{1/p}
\lesssim_pt_{2,X}^2\beta_{p,X}^6
\left\|f\right\|_{L^p(\Gm;X)}.
\]

For the family \((I_s^-)_s\), define
\[
\iota(n):=0\ominus n,\quad n\in\N,\qquad
\iota(A):=\{0\ominus n:n\in A\},\quad A\subset\N.
\]
For \(x=(x_j)_{j\ge0}\in\Gm\), let
\[
-x:=((-x_j)\bmod m_j)_{j\ge0},
\qquad (Rf)(x):=f(-x).
\]
Haar measure is invariant under the map \(x\mapsto-x\), so \(R\) is an isometric
involution on \(L^p(\Gm;X)\). Since
\[
\overline{\psi_n(x)}=\psi_{0\ominus n}(x),
\]
we have
\(
RP_AR=P_{\iota(A)}.
\)
Moreover,
\[
\iota(I_s^-)
=\iota(\zeta_s\oplus(w-1))\oplus[0,\ell_-).
\]
Thus \((\iota(I_s^-))_s\) is a pairwise disjoint family of Vilenkin
intervals. Applying Theorem~\ref{thm:oplus-equal-umd-type2-main}
to this family and to \(Rf\), and using the isometry of \(R\), gives
\[
\begin{aligned}
\left(\mathbb E_{\varepsilon}
\left\|
 \sum_s\varepsilon_sP_{I_s^-}f
\right\|_{L^p(\Gm;X)}^p
\right)^{1/p}
&=\left(\mathbb E_{\varepsilon}
\left\|
 R\sum_s\varepsilon_sP_{\iota(I_s^-)}Rf
 \right\|_{L^p(\Gm;X)}^p
\right)^{1/p}
\\
&\lesssim_{p}t_{2,X}^2\beta_{p,X}^6
 \left\|f\right\|_{L^p(\Gm;X)}.
\end{aligned}
\]
Since \(P_{I_s^{\oplus}}=P_{I_s^-}+P_{I_s^+}\), the triangle inequality
proves Corollary \ref{cor:oplus-common-interval}.
\end{proof}

\medskip

\noindent\textbf{Acknowledgements.} This work is partially supported by the National Natural Science Foundation of China (Grant Nos. 12325105, W2441002, W2611005).

\noindent\textbf{AI Statement.}
The authors acknowledge the use of ChatGPT for language
polishing, LaTeX editing, and exploratory mathematical discussions during the development and
preparation of this manuscript. Some ideas used in the proof-development process emerged during
interactions with GPT-5.6 Sol. Suggestions from these AI interactions were subsequently examined, reformulated, incorporated into the manuscript, and independently verified by
the authors. The authors assume full responsibility for all mathematical content in the final manuscript.

%\noindent \textbf{Data Availability.} Data sharing is not applicable to this article as no datasets were generated or analysed during the current study.

%\noindent \textbf{Conflicts of interest.} The authors declare that they have no conflicts of interest.

\end{document}